\documentclass[12pt]{amsart}
\usepackage{mystyle}
\begin{document}

\title{Phase diagram for intermittent maps}
\author{Daniel Coronel}
\address{Facultad de Matem\'aticas, Pontificia Universidad Cat\'olica de Chile, Campus San Joaqu\'in, Avenida Vicu\~{n}a Mackenna 4860,
  Santiago, Chile.}
\email{acoronel@uc.cl}
\author{Juan Rivera-Letelier}
\address{Department of Mathematics, University of Rochester. Hylan Building, Rochester, NY~14627, U.S.A.}
\email{riveraletelier@gmail.com}
\urladdr{\url{http://rivera-letelier.org/}}

\begin{abstract}
  This paper studies the phase diagram of a fixed \textsc{Pomeau}--\textsc{Manneville}, or intermittent, map under a varying \textsc{H{\"o}lder} continuous potential of fixed exponent.
  Depending on the exponent, there are either one or two thermodynamic phases.
  The \emph{intermittent} phase is characterized by a unique \textsc{Gibbs} state whose typical orbits alternate irregularly between time intervals spent near the neutral fixed point and chaotic excursions away from it.
  The \emph{stationary} phase is characterized by the \textsc{Dirac} mass at the neutral fixed point being the unique \textsc{Gibbs} state.
  Above a threshold exponent, every potential belongs to the intermittent phase, so the phase diagram is trivial.
  At or below the threshold, the phase transition locus is a co-dimension~$1$ topological submanifold separating the two phases, and it coincides with the set of potentials at which the pressure function fails to be real-analytic.

  This paper also investigates the relationships between the phase transition locus, (persistent) phase transitions in temperature, and ground states.
  It highlights a connection between phase transitions in temperature and ergodic optimization.
\end{abstract}
\date{\today}
\maketitle

\section{Introduction}
\label{s:introduction}

In the classical setting of uniformly hyperbolic maps and \textsc{H{\"o}lder} continuous potentials, there is a single thermodynamic phase: Every potential has a unique \textsc{Gibbs} state and typical orbits behave chaotically by the work of \textsc{Sina{\u{\i}}}, \textsc{Ruelle}, and \textsc{Bowen} \cite{Bow08,Rue76e,Sin72}.
The same conclusion holds for every sufficiently regular one-dimensional map satisfying a weak hyperbolicity assumption~\cite{LiRiv14b,LiRiv14a}.
This paper proves that for \textsc{Pomeau}--\textsc{Manneville}, or intermittent, maps the picture depends on the \textsc{H{\"o}lder} exponent of the potentials: There are either one or two thermodynamic phases.
The \emph{intermittent} phase is analogous to the classical, chaotic phase.
Every potential in this phase has a unique \textsc{Gibbs} state and typical orbits alternate irregularly between time intervals spent near the neutral fixed point and chaotic excursions away from it~\cite{PomMan80}.
The \emph{stationary} phase is characterized by the \textsc{Dirac} mass at the neutral fixed point being the unique \textsc{Gibbs} state.

More precisely, this paper studies the phase diagram of a fixed intermittent map under a varying \textsc{H{\"o}lder} continuous potential of fixed exponent.
The phase diagram divides the space of potentials into regions in which the corresponding \textsc{Gibbs} states exhibit qualitatively distinct behaviors.
Above a threshold exponent, every potential belongs to the intermittent phase, so the phase diagram is trivial~\cite{LiRiv14b,LiRiv14a}.
At or below the threshold, the phase transition locus is a co-dimension~$1$ topological submanifold separating the intermittent and stationary phases.
Moreover, this locus coincides with the set of potentials at which the pressure function fails to be real-analytic.

These results provide the first complete nontrivial phase diagram for a fixed dynamical system under a varying potential in a \textsc{Banach} space.
This approach is inspired by the global viewpoint of dynamical systems, initiated by \textsc{Smale} \cite{Sma67a} and developed by \textsc{Palis} \cite{Pal00b,Pal05} and others, and by phase diagrams in statistical mechanics.

This paper also investigates the relationships between the phase transition locus, (persistent) phase transitions in temperature, and ground states.
It further highlights a connection between phase transitions in temperature and ergodic optimization, established in the companion paper \cite{0CorInoRiv2504} for certain spaces of \textsc{H{\"o}lder}-continuous potentials.

The following notation and terminology allow a more precise formulation of the results.
For concreteness, fix throughout the rest of the paper~$\alpha$ in~$(0, +\infty)$ and the \emph{\textsc{Pomeau}--\textsc{Manneville}} or \emph{intermittent} map ${f \: [0, 1] \to [0, 1]}$, defined by
\begin{equation}
  \label{eq:1}
  f(x)
  \=
  \begin{cases}
    x(1 + x^{\alpha})
    & \text{if~$x(1 + x^\alpha) \le 1$;}
    \\
    x(1 + x^\alpha) - 1 & \text{otherwise}.
  \end{cases}
\end{equation}
It is a prototypical example of an interval map with a neutral periodic point.
It has been extensively studied in smooth ergodic theory, see, for example, \cite{BomCar23,BalTod16,BruTerTod19,ChaColRedVer09,CliTho13,CasVar13,GarIno22,GarIno24,Gou04a,Hol05,Hu04,Kor16,Klo19b,Klo20,LiRiv14a,LivSauVai99,MelTer12,Pia80,PomMan80,PreSla92,PolWei99,Sar01a,Sar02,ShevSt13,Tha80,You99b}.
The arguments apply with minor changes to variants of this map, like those studied in~\cite{You99b} and \cite[\S5]{LivSauVai99}.
For expository purposes, the map~$f$ has the advantage that its geometric potential~$- \log Df$ is \textsc{H{\"o}lder} continuous and thus belongs to (some of) the \textsc{Banach} spaces of potentials considered here.
However, this paper does not use this fact in the arguments.
It is possible to extend the results to maps with several neutral periodic points, as those considered by \textsc{Thaler} \cite{Tha80}.

Denote by~$C(\R)$ the space of continuous real-valued functions defined on~$[0, 1]$ endowed with the uniform norm~$\| \cdot \|$.
For each~$\gamma$ in~$(0, 1]$, denote by~$\Hg$ the space of those functions in~$C(\R)$ that are \textsc{H{\"o}lder} continuous of exponent~$\gamma$.
For each~$\varphi$ in~$\Hg$, put
\begin{equation}
  \label{eq:2}
  | \varphi |_{\gamma}
  \=
  \sup \left\{ \frac{|\varphi(x) - \varphi(y)|}{|x - y|^{\gamma}} \: x, y \in [0, 1] \text{ distinct} \right\}
  \text{ and }
  \| \varphi \|_{\gamma}
  \=
  \| \varphi \| + | \varphi |_{\gamma}.
\end{equation}

Denote by~$\sM$ the space of \textsc{Borel} probability measures on~$[0, 1]$ that are invariant by~$f$.
For each~$\nu$ in~$\sM$, denote by~$h_\nu$ the measure-theoretic entropy of~$\nu$.
For each~$\varphi$ in~$C(\R)$, the \emph{pressure~$P(\varphi)$ of~$f$ for the potential~$\varphi$} is defined by
\begin{equation}
  \label{eq:3}
  P(\varphi)
  \=
  \sup \left\{ h_\nu + \il \varphi \dd \nu \: \nu \in \sM \right\}.
\end{equation}
A measure at which this supremum is attained is a \emph{\textsc{Gibbs}} or \emph{equilibrium state of~$f$ for the potential~$\varphi$}.
There is at least~$1$ such measure because~$\sM$ is compact and the measure-theoretic entropy map is upper semi-continuous \cite[Lemma~2.6(5)]{0CorInoRiv2504}.
The pressure function ${P \: C(\R) \to \R}$ so defined is convex and \textsc{Lipschitz} continuous.

A measure~$\nu$ in~$\sM$ is \emph{exponentially mixing for~$f$}, if there are~$C$ in~$(0, +\infty)$ and~$\rho$ in~$(0, 1)$ such that, for all measurable bounded function ${\psi \: [0, 1] \to \R}$, \textsc{Lipschitz} continuous function ${\tpsi \: [0, 1] \to \R}$, and~$n$ in~$\N$,
\begin{equation}
  \label{eq:4}
  \left| \il \psi \circ f^n \cdot \tpsi \dd \nu - \il \psi \dd \nu \il \tpsi \dd \nu \right|
  \le
  C \| \psi \| \cdot | \tpsi |_1 \rho^n.
\end{equation}

\subsection{Intermittent and stationary phases}
\label{ss:phase-diagram}

Denote by~$\0$ the function defined on~$[0, 1]$ that vanishes identically.
Together with its unique \textsc{Gibbs} state for the potential~$\0$, the map~$f$ is measurably isomorphic to a \textsc{Bernoulli} process (\cref{l:mme}).
Due to the neutral fixed point at~$0$, typical orbits exhibit intermittent behavior in the sense of \textsc{Pomeau} and \textsc{Manneville} \cite{PomMan80}: They alternate irregularly between time intervals spent near~$0$ and chaotic excursions away from it.
This intermittent behavior persists for \textsc{H{\"o}lder} continuous potentials near~$\0$.
More precisely, for such a potential~$\varphi$, there is a \textsc{Gibbs} state of~$f$ that is exponentially mixing and has strictly positive measure-theoretic entropy.
See~\cite[Corollary~A.4]{0InoRiv26} and~\cite[Theorems~1.2 and~3.3]{Klo19b} for \textsc{Kloeckner}'s effective estimate of the size of this neighborhood of~$\0$ for a map similar to~$f$.
These strong stochastic properties of the \textsc{Gibbs} state guarantee that typical orbits with respect to this measure also exhibit intermittent behavior.
This motivates the following definition.

\begin{defi}
  \label{d:intermittent-phase}
  Let~$\gamma$ be in~$(0, 1]$.
  The \emph{intermittent phase~$\Ip$ of~$f$ in~$\Hg$} is the set of potentials in~$\Hg$ for which there is an exponentially mixing \textsc{Gibbs} state of~$f$ of strictly positive measure-theoretic entropy.
\end{defi}

For every~$\gamma$ in~$(0, 1]$, the set~$\Ip$ is nonempty because it contains (a neighborhood of)~$\0$.
Moreover, $\Ip$ is open, $P$ is real-analytic on~$\Ip$, and, for every~$\varphi$ in~$\Ip$, there is a unique \textsc{Gibbs} state of~$f$ for the potential~$\varphi$ \cite[Theorem~A.2]{0InoRiv26}.
See~\cref{ss:analiticity} for the notion of real-analyticity used in this paper.

In the case where~$\alpha < 1$ and ${\gamma > \alpha}$, the equality ${\Ip = \Hg}$ holds and thus~$P$ is real-analytic on~$\Hg$ \cite{LiRiv14b,LiRiv14a}.
If, in addition, ${\varphi(0) = \varphi(1)}$, then see also~\cite[Proposition~11]{GarIno24}, \textsc{Kloeckner}'s proof based on ideas from transport theory \cite[Theorem~A]{Klo20}, or the \cref{l:key} in~\cref{ss:lematta}.

In the remaining case where ${\gamma \le \alpha}$, the space~$\Hg$ contains the geometric potential~$-\log Df$.
It is the quintessential example of a potential for which every \textsc{Gibbs} state of~$f$ different from~$\delta_0$\footnote{Strictly speaking, $\delta_0$ is exponentially mixing for trivial reasons.} is at most subexponentially mixing \cite{Hu04,Gou04a,Sar02}.
That is, $-\log Df$ is outside~$\Ip$ and thus ${\Ip \neq \Hg}$.
In fact, ${\Hg \setminus \Ip}$ has nonempty interior.
Indeed, consider the function ${\omega_{\gamma} \: [0, 1] \to \R}$ given by ${\omega_{\gamma}(x) \= -x^{\gamma}}$.
For all sufficiently large constant~$\beta$ and potential~$\varphi$ in a neighborhood of~$\beta \omega_\gamma$ in~$\Hg$, the measure~$\delta_0$ is the unique \textsc{Gibbs} state of~$f$ for the potential~$\varphi$ (\cref{p:hook}(2)).
Thus, ${\beta \varphi}$ belongs to the interior of ${\Hg \setminus \Ip}$ and for it the system has a stationary behavior localized at the neutral fixed point~$0$.
This motivates the following definition.

\begin{defi}
  \label{d:stationary-phase}
  Let~$\gamma$ be in~$(0, 1]$.
  The \emph{stationary phase~$\Sp$ of~$f$ in~$\Hg$} is the subset of~$\Hg$ defined by
  \begin{equation}
    \label{eq:5}
    \Sp
    \=
    \ir(\{ \varphi \in \Hg \: \delta_0 \text{ is the unique \textsc{Gibbs} state of~$f$ for the potential } \varphi \}).
  \end{equation}
\end{defi}

By definition, $\Sp$ is open and disjoint from~$\Ip$.
When ${\gamma \le \alpha}$, the set~$\Sp$ is nonempty and unbounded because, for some~$\beta_0$ in~$(0, +\infty)$, it contains the ray~$(\beta \omega_\gamma)_{c \in (\beta_0, +\infty)}$.
Moreover, $\Ip$ is star-convex at~$\0$ and~$\Sp$ is convex (\cref{p:phases}(1, 2)).
In particular, $\Ip$ and~$\Sp$ are both connected.

\subsection{Phase transition locus and \textsc{Gibbs} phase rule}
\label{ss:pt-locus}

Fix~$\gamma$ in~$(0, 1]$.
If ${\gamma \le \alpha}$, then ${\Sp \neq \emptyset}$, and the sets~$\Ip$ and~$\Sp$ share the same boundary and their union is dense in~$\Hg$.
The same holds if~$\gamma > \alpha$ because in this case ${\Ip = \Hg}$ and ${\Sp = \emptyset}$.
So, the following definition arises naturally.

\begin{defi}
  \label{d:pt-locus}
  Let~$\gamma$ be in~$(0, 1]$.
  The \emph{phase transition locus~$\PTg$ of~$f$ in~$\Hg$} is the common boundary of the intermittent~$\Ip$ and stationary~$\Sp$ phases.
\end{defi}

The set~$\PTg$ is thus closed, has empty interior, and satisfies
\begin{equation}
  \label{eq:6}
  \PTg
  =
  \Hg \setminus (\Ip \cup \Sp).
\end{equation}

The following result characterizes~$\PTg$ analytically, as the singular set of~$P$, and topologically, as a submanifold of~$\Hg$.

\begin{theoalph}[Topological \textsc{Gibbs} Phase Rule]
  \label{t:pt-locus}
  For every~$\gamma$ in~$(0, 1]$, the phase transition locus~$\PTg$ coincides with the set of potentials in~$\Hg$ at which the pressure function~$P$ fails to be real-analytic.
  Moreover, ${\PTg \neq \emptyset}$ if and only if~$\gamma$ belongs to~$(0, \min \{ \alpha, 1 \}]$.
  In this case, ${\PTg}$ is a topological submanifold of~$\Hg$ of co-dimension~1 that is homeomorphic to a vector subspace of~$\Hg$ of co-dimension~1, but it is not an affine subspace of~$\Hg$.
\end{theoalph}

The phase transition locus is in fact linearly homeomorphic to the graph of a real-valued convex function, and every affine subspace contained in the phase transition locus is of infinite co-dimension (\cref{p:pt-locus}).

The following theorem is a finite-dimensional counterpart of \cref{t:pt-locus}, and it implies an analog of \textsc{Israel}'s formulation of the \textsc{Gibbs} phase rule \cite[Theorem~VI.3.3]{Isr79}.
For each~$\gamma$ in~$(0, 1]$, let~$\Ns$ be the closed vector subspace of~$\Hg$ defined by
\begin{equation}
  \label{eq:7}
  \Ns
  \=
  \left\{ \varphi \in \Hg \: \text{ for every~$\nu$ in~$\sM$, } \il \varphi \dd \nu = \varphi(0) \right\}.
\end{equation}
Potentials in the same translate of~$\Ns$ share the same \textsc{Gibbs} states and thus the sets~$\Ip$, $\Sp$, and~$\PTg$ are invariant under translations by elements of~$\Ns$.
That is,
\begin{equation}
  \label{eq:8}
  \Ip + \Ns
  =
  \Ip,
  \Sp + \Ns
  =
  \Sp,
  \text{ and }
  \PTg + \Ns
  =
  \PTg.
\end{equation}
For this reason, the result below restricts to subspaces of potentials not contained in~$\Ns$.

\begin{custtheo}{A'}[Geometric \textsc{Gibbs} Phase Rule]
  \label{t:pt-locus'}
  Let~$\gamma$ be in~$(0,\min\{\alpha,1\}]$, $k$ in~$\N$, and~$V$ a $k$\nobreakdash-dimensional vector subspace of~$\Hg$ satisfying ${V \not\subseteq \Ns}$.
  Denote by ${\pi \: \Hg \to \Hg/V}$ the quotient map and, for every~$\xi$ in~$\Hg/V$, put ${V_\xi \= \pi^{-1}(\xi)}$.
  Then, for every~$\xi$ in an open dense subset of~$\Hg/V$, either ${V_\xi \subseteq \Ip}$ or
  \begin{equation}
    \label{eq:9}
    V_\xi \cap \Ip
    \neq
    \emptyset,
    V_\xi \cap \Sp
    \neq
    \emptyset,
  \end{equation}
  and~$V_\xi \cap \PTg$ is a \textsc{Lipschitz} manifold of dimension~${k - 1}$.
  In particular, ${\HD(V_\xi \cap \PTg)}$ equals~$0$ in the former case and ${k - 1}$ in the latter.
\end{custtheo}

\cref{t:pt-locus} has a natural application to 1-parameter families of potentials.
Let~$I$ be an open interval of~$\R$ and~$(\varphi_\tau)_{\tau \in I}$ a real-analytic family of potentials in~$\Hg$.
\cref{t:pt-locus} implies that, for every parameter~$\tau_0$ at which the function ${\tau \mapsto P(\varphi_\tau)}$ fails to be real-analytic, the potential~$\varphi_{\tau_0}$ must be in~$\PTg$.
In other words, every phase transition parameter occurs at a potential in~$\PTg$.
In the reverse direction, if~$\tau_0$ is a parameter such that~$\varphi_\tau$ belongs to~$\Ip$ if ${\tau < \tau_0}$ and outside~$\Ip$ if ${\tau \ge \tau_0}$, then ${\tau \mapsto P(\varphi_\tau)}$ fails to be real-analytic at~$\tau_0$ (\cref{l:singular-family}).
Together with the fact that~$\Ip$ is star-convex at~$\0$ (\cref{p:phases}(1)) and equal to ${\{\varphi \in \Hg \: P(\varphi) > \varphi(0) \}}$ \cite[Theorem~A.2]{0InoRiv26}, this implies the following corollary as an immediate consequence.

\begin{coro}[Phase transitions in temperature]
  \label{c:pt-in-temperature}
  For every \textsc{H{\"o}lder} continuous potential~$\varphi$, precisely one of the following properties holds.
  \begin{enumerate}
  \item
    The family~$(\beta \varphi)_{\beta \in (0, +\infty)}$ is contained in~$\Ip$ and the function ${\beta \mapsto P(\beta \varphi)}$ is real-analytic and strictly larger than ${\beta \mapsto \beta \varphi(0)}$ on~$(0, +\infty)$.
  \item
    There is~$\beta_*$ in~$(0, +\infty)$ such that, for every~$\beta$ in~$(0, +\infty)$, the potential~$\beta \varphi$ belongs to~$\Ip$ if~$\beta$ belongs to~$(0, \beta_*)$ and is outside~$\Ip$ if~$\beta$ belongs to~$[\beta_*, +\infty)$.
    Thus, $\beta_* \varphi$ belongs to~$\PTg$ and the function ${\beta \mapsto P(\beta \varphi)}$ is real-analytic and strictly larger than ${\beta \mapsto \beta \varphi(0)}$ on~$(0, \beta_*)$, it coincides with ${\beta \mapsto \beta \varphi(0)}$ on~$[\beta_*, +\infty)$, and it only fails to be real-analytic on~$(0, +\infty)$ at~$\beta_*$.
  \end{enumerate}
\end{coro}

If property~2 in \cref{c:pt-in-temperature} holds, then the potential~$\varphi$ undergoes a \emph{phase transition in temperature} and~$\beta_*$ is the \emph{phase transition parameter of~$\varphi$}.
This terminology comes from statistical mechanics, where~$\beta$ is interpreted as the inverse temperature.
Thus, a \textsc{H{\"o}lder} continuous potential~$\varphi$ undergoes a phase transition in temperature if and only if the function ${\beta \mapsto P(\beta \varphi)}$ fails to be real-analytic on~$(0, +\infty)$.
\cref{c:pt-in-temperature} implies that these equivalent conditions hold if and only if at low temperatures~$\delta_0$ is a \textsc{Gibbs} state of~$f$ for~$\varphi$.

The geometric potential~$-\log Df$ is a well-known example of a potential having a phase transition in temperature.
See, for example, \cref{p:geometric-potential}, \cite{PreSla92}, \cite[Proposition~4.6]{0CorInoRiv2504}, as well as \cite[Theorem~B]{BomCar23} for an analogous result for local diffeomorphisms of the circle.
For potentials that have a precise asymptotic behavior near~$0$ and are sufficiently regular, phase transitions in temperature were studied by \textsc{Sarig} \cite[Proposition~1]{Sar01a} and in the companion paper~\cite{0CorInoRiv2504}.

\subsection{On the persistence of phase transitions in temperature}
\label{ss:(non)persistent-pt}

The next result is reminiscent of the \textsc{Peierls} condition for contour models, which ensures that ground states persist as \textsc{Gibbs} states at low temperatures \cite{Pei36,Sin82}.
In various situations of interest, like the \textsc{Ising} model on the square lattice~\cite{Ons44}, this leads to the existence of phase transitions \cite[Chapter~6]{Geo11}.
For intermittent maps and a \textsc{H{\"o}lder} continuous potential~$\varphi$, \cref{t:persistence-locus}, below, asserts that the condition of~$\delta_0$ being robustly a ground state implies that~$\delta_0$ is also a \textsc{Gibbs} state at low temperatures.
As a consequence, such a potential~$\varphi$ undergoes a phase transition in temperature.

The following notation and terminology allow a more precise formulation of this result.
Let~$\varphi$ be a continuous potential.
A \emph{maximizing measure} or \emph{ground state} of~$f$ for~$\varphi$ is a measure~$\nu$ in~$\sM$ satisfying
\begin{equation}
  \label{eq:10}
  \il \varphi \dd \nu
  =
  \sup_{\nu' \in \sM} \il \varphi \dd \nu'.
\end{equation}
For each~$\gamma$ in~$(0, 1]$, put
\begin{equation}
  \label{eq:11}
  \Sc
  \=
  \{ \varphi \in \Hg \: \delta_0 \text{ is a ground state of~$f$ for the potential~$\varphi$} \}.
\end{equation}
It is a closed convex cone whose interior is nonempty because it contains, for example, $\omega_{\gamma}$.
In particular, ${\Sc}$ has a dense interior.
When~$\gamma$ belongs to~$(0, \min\{\alpha, 1\}]$, the inclusion ${\Sp \subseteq \ir(\Sc)}$ holds (\cref{p:nonpersistent-pt}).

Recall from~\cref{ss:pt-locus} that, for a \textsc{H{\"o}lder} continuous potential~$\varphi$ having a phase transition in temperature, the phase transition parameter of~$\varphi$ is the unique parameter~$\beta_*$ in~$(0, +\infty)$ at which the function ${\beta \mapsto P(\beta \varphi)}$ fails to be real-analytic.
For each~$\gamma$ in~$(0, 1]$ such that~$\Hg$ contains~$\varphi$, put ${\ct(\varphi) \= \beta_*}$ in this case and ${\ct(\varphi) \= +\infty}$ otherwise.

\begin{defi}
  \label{d:persistent-pt}
  Let~$\gamma$ be in~$(0, 1]$.
  A potential~$\varphi$ in~$\Hg$ undergoes a \emph{persistent phase transition in temperature in~$\Hg$} if ${\ct(\varphi) < +\infty}$ and the function~$\ct$ is continuous at~$\varphi$.
\end{defi}

\begin{theoalph}[From zero to low temperatures]
  \label{t:persistence-locus}
  Fix~$\gamma$ in~$(0, \min\{\alpha, 1\}]$.
  Then, for every potential~$\varphi$ in~$\Hg$, the following properties are equivalent.
  \begin{enumerate}
  \item
    $\varphi$ undergoes a persistent phase transition in temperature in~$\Hg$.
  \item
    $\delta_0$ is, robustly in~$\Hg$, the unique ground state of~$f$ for~$\varphi$.
  \item
    $\varphi$ belongs to~$\ir(\Sc)$.
  \end{enumerate}
\end{theoalph}

The equivalence of properties~$1$ and~$2$ in \cref{t:persistence-locus} highlights the connection between phase transitions in temperature and ergodic optimization, as observed in \cite[Theorems~1 and~3 and Corollaries~1.1 and~2.11]{0CorInoRiv2504} for certain spaces of \textsc{H{\"o}lder} continuous potentials introduced by \textsc{Sarig} in \cite[\S4]{Sar01a}.
For background on ergodic optimization, see the survey article~\cite{Jen19}.

By \cref{c:pt-in-temperature}, every potential in~$\Hg$ undergoing a phase transition in temperature belongs to~$\Sc$,
\begin{equation}
  \label{eq:12}
  \ct^{-1}((1, +\infty])
  =
  \Ip,
  \text{ and }
  \ct^{-1}(1)
  \subseteq
  \PTg.
\end{equation}
Thus, the following corollary is a direct consequence of \cref{t:persistence-locus}.

\begin{coro}
  \label{c:persistence-locus}
  For every~$\gamma$ in~$(0, \min\{\alpha, 1\}]$, the function~$\ct$ is constant equal to~$+\infty$ outside~$\Sc$ and it is finite and continuous on the nonempty set~$\ir(\Sc)$.
\end{coro}

A natural question that these results raise is whether every phase transition in temperature is persistent.
In some cases, the geometric potential~$-\log Df$ provides a negative answer.
Recall that this potential undergoes a phase transition in temperature and that its phase transition parameter equals~$1$ (\cref{p:geometric-potential}).
This potential is \textsc{H{\"o}lder} continuous of exponent~$\min \{ \alpha, 1 \}$, so, for every~$\gamma$ in~$(0, \min \{ \alpha, 1 \}]$, it belongs to~$\Hg$.
In the case where ${\gamma < \alpha}$, this paper shows that the phase transition in temperature of~$-\log Df$ is nonpersistent in~$\Hg$, and thus~$\ct$ is discontinuous at this potential.
Moreover, ${\PTg}$ fails to be a co-dimension~$1$ real-analytic subset of~$\Hg$ at~$-\log Df$.
See \cref{p:nonpersistent-pt} for a precise statement of these results, and \cref{figure}(A) below.

The following result asserts that, in the remaining case where~$\alpha$ belongs to~$(0, 1]$ and ${\gamma = \alpha}$, every phase transition in temperature is persistent.
It also characterizes potentials undergoing a phase transition in temperature.
See also \cref{p:persistent-pt} and \cref{figure}(B) below.
Denote by~$x_1$ the unique discontinuity of~$f$, so ${f \: [0, x_1] \to [0, 1]}$ is a diffeomorphism.
For each integer~$j$ satisfying ${j \ge 2}$, define~$x_j$ inductively as the unique preimage of~$x_{j - 1}$ in~$[0, x_1]$.

\begin{theoalph}
  \label{t:alpha-rigidity}
  Suppose~$\alpha$ belongs to~$(0, 1]$.
  Then, for every~$\varphi$ in~$\Ha$, the following properties are equivalent.
  \begin{enumerate}
  \item
    $\varphi$ undergoes a persistent phase transition in temperature in~$\Ha$.
  \item
    $\varphi$ undergoes a phase transition in temperature.
  \item
    $\delta_0$ is the unique ground state of~$f$ for the potential~$\varphi$, and there are~$C$ and~$\delta$ in~$(0, +\infty)$ such that, for every~$n$ in~$\N$,
    \begin{equation}
      \label{eq:13}
      \exp(S_n(\varphi)(x_n) - n \varphi (0))
      \le
      \frac{C}{n^{\delta}}.
    \end{equation}
  \end{enumerate}
  Moreover, the function~$\cta$ is continuous,
  \begin{equation}
    \label{eq:14}
    \cta^{-1}(1)
    =
    \PTa
    \subseteq
    \ir(\Sca),
    \text{ and }
    \cta^{-1}((0, +\infty))
    =
    \ir(\Sca).
  \end{equation}
\end{theoalph}

\begin{figure}[htbp]
  \centering
  \begin{subfigure}[b]{0.45\textwidth}
    \centering
    \includegraphics[width=\textwidth]{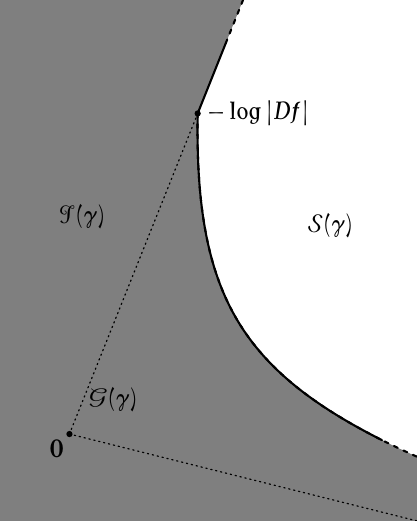}
    \caption{Case ${\gamma < \alpha}$.}
  \end{subfigure}
  \hfill
  \begin{subfigure}[b]{0.45\textwidth}
    \centering
    \includegraphics[width=\textwidth]{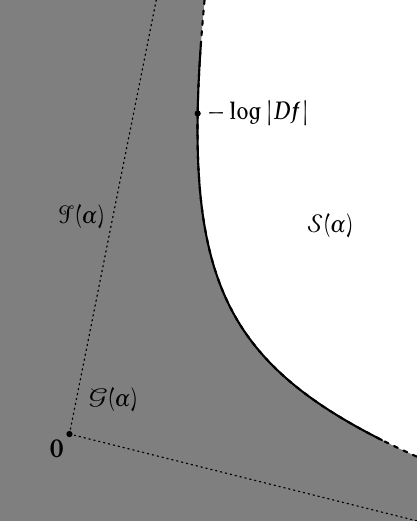}
    \caption{Case ${\alpha \le 1}$ and ${\gamma = \alpha}$.}
  \end{subfigure}
  \caption{Phase diagram and the stationary cone.}
  \label{figure}
\end{figure}

The following corollary is a direct consequence of Theorems~\ref{t:pt-locus} and~\ref{t:alpha-rigidity}.

\begin{coro}
  \label{c:only-temperature-pt}
  Suppose~$\alpha$ belongs to~$(0, 1]$, let~$I$ be an open subinterval of~$\R$ containing~$0$, and let~$(\varphi_\tau)_{\tau \in I}$ be a real-analytic family of potentials in~$\Ha$ having a phase transition at~$0$.
  Then, $\varphi_0$ undergoes a phase transition in temperature, ${\cta(\varphi_0) = 1}$, and~$\varphi_0$ belongs to~$\PTa$.
\end{coro}

For a concrete example showing that the analogous statement is false when ${\gamma < \alpha}$, see \cref{eq:non-temperature-pt}.

As mentioned above, for every~$\gamma$ in~$(0, 1]$ satisfying ${\gamma < \alpha}$, the set~$\PTg$ fails to be a co-dimension~$1$ real-analytic subset of~$\Hg$.
Thus, the following question arises naturally.

\begin{ques}
  If~$\alpha$ belongs to~$(0, 1]$, is~$\PTa$ a co-dimension~$1$ real-analytic subset of~$\Ha$?
\end{ques}

\subsection{Notes and references}
\label{ss:notes-references}

The intermittent behavior of a positive entropy \textsc{Gibbs} state reflects the concentration near the neutral fixed point~$0$ of its density with respect to the conformal measure.
See, for example, \cite{LivSauVai99,Tha80}.

In the definition~\cref{eq:5} of the stationary phase~$\Sp$, uniqueness can be dropped.
More precisely, the set obtained by dropping uniqueness on the right side of~\cref{eq:5} is also equal to~$\Sp$ by~\cref{eq:28} in \cref{p:phases}.
However, taking the interior is necessary.
Indeed, for all~$\gamma$ in~$(0, 1]$ satisfying ${\gamma < \alpha}$ and~$\beta$ in~$(1, +\infty)$, the measure~$\delta_0$ is the unique \textsc{Gibbs} state for the potential~$-\beta \log Df$ (\cref{c:states} and \cref{p:geometric-potential}), but~$-\beta \log Df$ is outside~$\Sp$ (\cref{p:nonpersistent-pt}(2)).

For each \textsc{H{\"o}lder} continuous potential~$\varphi$ undergoing a phase transition in temperature, $\delta_0$ is the unique ground state of~$f$ for the potential~$\varphi$ (Corollaries~\ref{c:pt-in-temperature} and~\ref{c:states}).
The following is a counterexample for the reverse implication.
Choose~$\gamma$ in~$(\alpha, +\infty)$ and let ${\omega_{\gamma} \: [0, 1] \to \R}$ be the function defined by ${\omega_{\gamma}(x) \= -x^{\gamma}}$.
It is \textsc{H{\"o}lder} continuous of exponent~$\min \{ 1, \gamma \}$.
Clearly, $\delta_0$ is the unique ground state of~$f$ for the potential~$\omega_{\gamma}$.
However, $\omega_{\gamma}$ has no phase transition in temperature by \cref{p:hook}.

The companion paper~\cite{0CorInoRiv2504} develops the thermodynamic formalism, for each~$\gamma$ in~$(0, 1]$, for the potentials in~$\Hg$ introduced by \textsc{Sarig} in~\cite[\S4]{Sar01a}.
These potentials form a \textsc{Banach} subspace~$\Sspace$ of~$\Hg$.
Since the inclusion of this subspace is continuous and thus real-analytic, the phase diagram on~$\Sspace$, including analogs of Theorems~\ref{t:pt-locus} and~\ref{t:pt-locus'}, follows from that of~$\Hg$.
The paper~\cite{0CorInoRiv2504} proves results analogous to Theorems~\ref{t:persistence-locus} and~\ref{t:alpha-rigidity} using the notion of ``robust phase transition in temperature'' \cite[Theorem~3 and Corollary~1.3]{0CorInoRiv2504}.
For~$\gamma$ in~$(0,1]$, this notion can be formulated in the space~$\Hg$ as follows: A potential~$\varphi$ in~$\Hg$ exhibits a \emph{robust phase transition in temperature in $\Hg$} if every potential sufficiently close to~$\varphi$ in~$\Hg$ undergoes a phase transition in temperature.
Clearly, every phase transition in temperature that is persistent in~$\Hg$ is also robust in~$\Hg$.
\cref{c:persistence-locus} shows that these notions coincide in~$\Hg$.

The inducing scheme has been successful in establishing fundamental properties of intermittent maps.
However, for~$\gamma$ in~$(0, 1]$ satisfying ${\gamma < \alpha/(\alpha + 1)}$, the inducing scheme approach breaks down for potentials in~$\Hg$ because the distortion is, in general, unbounded.
This paper thus avoids the inducing scheme approach entirely.

The results of this paper contrast with several genericity and flexibility results.
For circle maps with a neutral fixed point that are uniformly expanding away from it, \textsc{Kloeckner} proved that \textsc{H{\"o}lder} potentials whose transfer operators have a spectral gap are dense in suitable weaker \textsc{H{\"o}lder} topologies \cite[Theorem~C]{Klo20}.
For transitive local diffeomorphisms of the circle, \textsc{Bomfim} and \textsc{Fernandes} proved the existence of an open dense subset of continuous potentials, in the uniform topology, whose \textsc{H{\"o}lder} continuous elements have no phase transition and whose transfer operators have a spectral gap \cite[Theorem~1.4]{BomFer26}.
In both cases, genericity holds in a topology weaker than the \textsc{H{\"o}lder} topology considered here, so neither result implies genericity of the spectral gap, or of the absence of phase transitions, among \textsc{H{\"o}lder} potentials.
Similarly, for countable \textsc{Markov} shifts, \textsc{Cyr} and \textsc{Sarig} proved that strongly positive recurrent potentials form an open and dense set in suitable symbolic topologies \cite[Theorem~2.2]{CyrSar09}.
In larger spaces of potentials, thermodynamic behavior is far more flexible.
\textsc{Kucherenko} and \textsc{Quas} showed that, for one-dimensional compact symbolic systems, essentially any convex, \textsc{Lipschitz}, and asymptotically linear function arises as a pressure function \cite[Theorem~1]{KucQua22}.
In particular, continuous potentials may have phase transitions at a countable dense set of inverse temperatures.
By contrast, every \textsc{H{\"o}lder} continuous potential in the present setting has at most one phase transition in temperature.

\subsection{Organization}
\label{ss:organization}

After some preliminaries in~\cref{s:preliminaries}, \cref{s:phase-diagram} explores the phase diagram.
\cref{ss:phases} examines the intermittent and stationary phases and proves \cref{t:pt-locus'}.
\cref{ss:receding-to-zero-temperature} analyzes the geometry at infinity of both the stationary phase and the phase transition locus.
Finally, \cref{ss:proof-pt-locus} uses this analysis to establish fundamental properties of the phase transition locus (\cref{p:pt-locus}) and to derive \cref{t:pt-locus}.

\cref{s:pt} studies the persistence of phase transitions in temperature.
\cref{ss:zero-to-positive} relates, for each~$\gamma$ in~$(0, \min\{\alpha, 1\}]$, the set~$\Sp$ to the cone~$\Sc$ and proves \cref{t:persistence-locus}.
\cref{ss:nonpersistent-pt} states and proves the results when ${\gamma < \alpha}$ and provides an example of a real-analytic family of potentials having a phase transition at a potential where no phase transition in temperature takes place (\cref{eq:non-temperature-pt}).
\cref{ss:persistent-pt} characterizes the potentials that undergo a phase transition in temperature in the remaining case, where ${\alpha \le 1}$ and ${\gamma = \alpha}$ (\cref{p:persistent-pt}).
This characterization is the main ingredient in the proof of \cref{t:alpha-rigidity}, which is given in~\cref{ss:proof-alpha-rigidity}.

\section{Preliminaries}
\label{s:preliminaries}

Let~$X$ be a topological space.
For a subset~$S$ of~$X$, denote by~$\ir(S)$ and~$\cl(S)$ the interior and the closure of~$S$, respectively.
A subset of~$X$ is a \emph{regular open set}, if it equals the interior of its closure.

For each subset~$K$ of~$\R$, denote by~$\HD(K)$ its \textsc{Hausdorff} dimension.

For each~$\nu$ in~$\sM$, denote by~$\supp(\nu)$ its support.

Denote by~$C(\C)$ the space of continuous complex-valued functions defined on~$[0, 1]$ and for~$\varphi$ in~$C(\C)$, put
\begin{equation}
  \label{eq:15}
  \| \varphi \|
  \=
  \sup_{[0, 1]} |\varphi|.
\end{equation}
Furthermore, denote by~$C(\R)$ the subspace of~$C(\C)$ of real-valued functions and by~$\0$ the function in~$C(\R)$ that is constant equal to~$0$.

For each~$\gamma$ in~$(0, 1]$, denote by~$\HgC$ the space of those functions in~$C(\C)$ that are \textsc{H{\"o}lder} continuous of exponent~$\gamma$.
For each~$\varphi$ in~$\HgC$, put
\begin{equation}
  \label{eq:16}
  | \varphi |_{\gamma}
  \=
  \sup \left\{ \frac{|\varphi(x) - \varphi(y)|}{|x - y|^{\gamma}} \: x, y \in [0, 1] \text{ distinct} \right\}
  \text{ and }
  \| \varphi \|_{\gamma}
  \=
  \| \varphi \| + | \varphi |_{\gamma}.
\end{equation}
Furthermore, denote by~$\Hg$ the subspace of~$\HgC$ of real-valued functions.
The normed spaces $(\Hg, \| \cdot \|_{\gamma})$ and~$(\HgC, \| \cdot \|_{\gamma})$ are both \textsc{Banach}.

\subsection{Holomorphic and real-analytic functions and subsets}
\label{ss:analiticity}
Throughout this section, fix~$\gamma$ in~$(0, 1]$.

For each open subset~$\Delta$ of~$\C$, a function ${H \: \Delta \to \HgC}$ is \emph{holomorphic} if, for every~$z$ in~$\Delta$, there is~$\varphi_z$ in~$\HgC$ such that
\begin{equation}
  \label{eq:17}
  \lim_{h \to 0} \left\| \varphi_z - \frac{H(z + h) - H(z)}{h} \right\|_{\gamma}
  =
  0.
\end{equation}
Moreover, for each open subset~$\cU$ of~$\HgC$, a function ${h \: \cU \to \C}$ is \emph{holomorphic} if, for all open subset~$\Delta$ of~$\C$ and holomorphic function ${H \: \Delta \to \cU}$, the composition~$h \circ H$ is holomorphic.

A real-valued function defined on an open subset of~$\Hg$ is \emph{real-analytic}, if it extends to a complex-valued holomorphic function defined on an open subset of~$\HgC$.
Similarly, a function defined on an open subset of~$\R$ and taking values in~$\Hg$ is \emph{real-analytic}, if it extends to a function defined on an open subset of~$\C$ and taking values in~$\HgC$.

A subset~$\cS$ of~$\Hg$ is a \emph{co-dimension~$1$ real-analytic subset of~$\Hg$} if, for every~$\varphi$ in~$S$, there is an open neighborhood~$\cU$ of~$\varphi$ and a nonconstant real-analytic function ${\cU \to \R}$ vanishing precisely on ${\cU \cap \cS}$.

\begin{lemm}
  \label{l:identity}
  Fix~$\gamma$ in~$(0, 1]$ and let~$\cS$ be a co-dimension~$1$ real-analytic subset of~$\Hg$ that is closed.
  Moreover, let~$I$ be an open interval of~$\R$ and ${A \: I \to \Hg}$ a real-analytic function.
  Then, either ${A(I) \subseteq S}$ or every point of~$A^{-1}(S)$ is isolated in~$I$.
\end{lemm}

\begin{proof}
  Put ${K \= A^{-1}(\cS)}$.
  Since~$\cS$ is closed by hypothesis, $K$ is closed in~$I$.

  The first step is to prove that every accumulation point~$\tau_0$ of~$K$ in~$I$ belongs to~$\ir(K)$.
  Let~$\cU$ be an open neighborhood of~$A(\tau_0)$ in~$\Hg$ and ${a \: \cU \to \R}$ a real-analytic function vanishing precisely on ${\cU \cap \cS}$.
  Furthermore, let~$I_0$ be an open subinterval of~$I$ containing~$\tau_0$ such that ${A(I_0) \subseteq \cU}$.
  Then, the composition~$a \circ A|_{I_0}$ vanishes on ${K \cap I_0}$ and it extends to a holomorphic function defined on a neighborhood of~$I_0$ in~$\C$.
  Combined with the hypothesis that~$\tau_0$ is not isolated in ${K \cap I_0}$ and the Identity Principle, this implies that~$a \circ A|_{I_0}$ vanishes on all of~$I_0$.
  In particular, ${I_0 \subseteq K}$ and~$\tau_0$ belongs to~$\ir(K)$.

  To prove the lemma, suppose~$K$ contains a non-isolated point~$\tau_0$ in~$I$.
  By the above, ${\tau_0}$ belongs to~$\ir(K)$.
  Let~$J_0$ be the connected component of~$\ir(K)$ containing~$\tau_0$.
  If~$J_0$ were different from~$I$, then one of the endpoints of~$J_0$ would belong to~$I$.
  Being an accumulation point of~$K$, this point would belong to~$\ir(K)$, which is absurd.
  This contradiction proves ${J_0 = I}$ and thus ${K = I}$.
\end{proof}

\subsection{Dynamics and thermodynamic formalism of intermittent maps}
\label{ss:lematta}
The function~$\log Df$ is strictly increasing and \textsc{H{\"o}lder} continuous of exponent~$\min \{1, \alpha\}$.

Put ${x_0 \= 1}$ and recall from~\cref{ss:pt-locus} that~$x_1$ denotes the unique discontinuity of~$f$.
So, ${f(x_1) = x_0}$ and ${f \: [0, x_1] \to [0, 1]}$ is a diffeomorphism.
Moreover, for each integer~$j$ satisfying ${j \ge 2}$, the point~$x_j$ equals~$f|_{[0, x_1]}^{-(j - 1)}(x_1)$.
For every~$j$ in~$\N_0$, put ${J_j \= (x_{j + 1}, x_j]}$.
So, ${f(x_{j + 1}) = x_j}$ and ${f(J_{j + 1}) = J_j}$.

\begin{lemm}[\textcolor{black}{\cite[Lemma~2.1]{0CorInoRiv2504}}]
  \label{l:neutral-branch}
  The equality
  \begin{equation}
    \label{eq:18}
    \lim_{n \to +\infty} n \cdot x_n^{\alpha}
    =
    \frac{1}{\alpha}
  \end{equation}
  holds.
  Moreover, there is~$\varepsilon_0$ in~$(0, 1)$ such that, for all~$n$ in~$\N$ and~$x$ in~$J_n$,
  \begin{equation}
    \label{eq:19}
    (1 + \varepsilon_0 n)^{\frac{1}{\alpha} + 1}
    \le
    Df^n(x)
    \le
    (1 + \varepsilon_0^{-1} n)^{\frac{1}{\alpha} + 1}.
  \end{equation}
\end{lemm}

Let~$\sP$ be the partition of~$[0, 1]$ defined by ${\cP \= \{ [0, x_1], (x_1, 1] \}}$ and, for every~$n$ in~$\N$, put ${\cP_n \= \bigvee_{k = 0}^{n - 1} \cP}$.

\begin{lemm}[\textcolor{black}{\cite[Proposition~2.9]{0CorInoRiv2504}}]
  \label{l:pressure}
  For every continuous potential~$\varphi$,
  \begin{equation}
    \label{eq:intermittent-tree}
    P(\varphi)
    =
    \lim_{n \to + \infty} \frac{1}{n} \log \sum_{Q \in \cP_n} \sup_Q \exp(S_n(\varphi))
  \end{equation}
  and, for every~$x$ in~$(0, 1]$,
  \begin{equation}
    \label{eq:20}
    P(\varphi)
    =
    \lim_{n \to + \infty} \frac{1}{n} \log \sum_{x' \in f^{-n}(x)} \exp(S_n(\varphi)(x')).
  \end{equation}
\end{lemm}

\begin{lemm}
  \label{l:mme}
  The equality
  \begin{equation}
    \label{eq:21}
    \sup_{\nu \in \sM} h_\nu
    =
    \log 2
  \end{equation}
  holds and there is a unique measure~$\nu_{\max}$ realizing this supremum.
  Furthermore, $(f, \nu_{\max})$ is measurably isomorphic to a \textsc{Bernoulli} process.
  That is, there is a measurable map ${[0, 1] \to \{ 0, 1 \}^{\N_0}}$ mapping~$\nu_{\max}$ to the \textsc{Bernoulli} measure with weights~$\frac{1}{2}$ and~$\frac{1}{2}$ and conjugating~$f$ to the shift map on a set of full measure.
\end{lemm}

\begin{proof}
  Denote by ${\sigma \: \{0,1\}^{\N_0} \to \{0,1\}^{\N_0}}$ the shift map, by~$\sM_\sigma$ the space of Borel probability measures on $\{0,1\}^{\N_0}$ invariant by~$\sigma$.
  Recall that the \textsc{Bernoulli} measure~$\rho$ with weights~$\frac{1}{2}$ and~$\frac{1}{2}$ on~$\{0,1\}^{\N_0}$ is the unique measure of maximal entropy for~$\sigma$.

  By \cite[Lemma~2.6(4, 5)]{0CorInoRiv2504},
  \begin{equation}
    \sup_{\mu \in \sM} h_\mu
    =
    \sup_{\mu \in \sM_{\sigma} } h_\mu(\sigma)
    =
    \log 2,
  \end{equation}
  there is a unique measure~$\nu_{\max}$ realizing the first supremum above, and~$(\sigma, \rho)$ and $(f, \nu_{\max})$ are measurably-isomorphic.
\end{proof}

The following lemma originated in~\cite{InoRiv12} in the context of complex rational maps.
See \cite[\S3]{0CorInoRiv2504} for a proof, and~\cite{LiH15} for a proof for multimodal maps.

\begin{custtheobis}{Lemma}
  \label{l:key}
  For all \textsc{H{\"o}lder} continuous potential~$\varphi$ and~$\nu$ in~$\sM$ distinct from~$\delta_0$,
  \begin{equation}
    \label{eq:22}
    P(\varphi)
    >
    \il \varphi \dd \nu.
  \end{equation}
\end{custtheobis}

\begin{coro}
  \label{c:states}
  Let~$\varphi$ be a \textsc{H{\"o}lder} continuous potential such that, for some~$\beta_*$ in~$(0, +\infty)$, the equality ${P(\beta_*\varphi) = \beta_* \varphi(0)}$ holds.
  Then, $\delta_0$ is the unique ground state of~$f$ for the potential~$\varphi$ and, for every~$\beta$ in~$(\beta_*, +\infty)$, the measure~$\delta_0$ is the unique \textsc{Gibbs} state of~$f$ for the potential~$\beta\varphi$.
\end{coro}

\begin{proof}
  From
  \begin{equation}
    \label{eq:23}
    \beta_* \varphi(0)
    =
    P(\beta_* \varphi)
    \ge
    \beta_* \sup_{\nu \in \sM} \il \varphi \dd \nu
    \ge
    \beta_* \varphi(0),
  \end{equation}
  it follows that~$\delta_0$ is a ground state of~$f$ for the potential~$\varphi$.
  If~$\nu_0$ is another ground state of~$f$ for~$\varphi$, then
  \begin{equation}
    \label{eq:24}
    \il \varphi \dd \nu_0
    =
    \sup_{\nu \in \sM} \il \varphi \dd \nu
    =
    \varphi(0).
  \end{equation}
  Together with the \cref{l:key}, this implies ${\nu_0 = \delta_0}$.

  To prove the second assertion, let~$\beta$ be in~$(\beta_*, +\infty)$.
  Since~$\sM$ is compact and the measure-theoretic entropy map is upper semi-continuous \cite[Lemma~2.6(5)]{0CorInoRiv2504}, there is a \textsc{Gibbs} state~$\nu$ of~$f$ for the potential~$\beta \varphi$.
  Then,
  \begin{equation}
    \label{eq:26}
    \beta \varphi(0)
    =
    \frac{\beta}{\beta_*} P(\beta_* \varphi)
    \ge
    \frac{\beta}{\beta_*} h_\nu + \beta \il \varphi \dd \nu
    =
    \frac{\beta - \beta_*}{\beta_*} h_\nu + P(\beta \varphi)
    \ge
    \frac{\beta - \beta_*}{\beta_*} h_\nu + \beta \varphi(0)
  \end{equation}
  and thus ${h_\nu = 0}$ and ${P(\beta \varphi) = \beta \varphi(0)}$.
  The latter implies that~$\delta_0$ is a \textsc{Gibbs} state of~$f$ for the potential~$\beta\varphi$.
  Together with the \cref{l:key} and the fact that~$\nu$ is a \textsc{Gibbs} state of~$f$ for the potential~$\beta\varphi$, the former implies ${\nu = \delta_0}$.
  That is, $\delta_0$ is the unique \textsc{Gibbs} state of~$f$ for the potential~$\beta\varphi$.
\end{proof}

\section{Phase diagram}
\label{s:phase-diagram}
This section explores the phase diagram.
\cref{ss:phases} examines the intermittent and stationary phases (\cref{p:phases}) and proves \cref{t:pt-locus'}.
\cref{ss:receding-to-zero-temperature} analyzes the geometry at infinity of both the stationary phase and the phase transition locus (\cref{p:receding-to-zero-temperature}).
Building on this analysis, \cref{ss:proof-pt-locus} establishes fundamental facts about the structure of the phase transition locus (\cref{p:pt-locus}) and derives \cref{t:pt-locus}.

\subsection{Intermittent and stationary phases}
\label{ss:phases}
This section proves the following proposition.

\begin{prop}
  \label{p:phases}
  Let~$\gamma$ be in~$(0, 1]$.
  Then, $\Ip$ contains~$\0$ and thus ${\Ip \neq \emptyset}$.
  In the case where~$\alpha < 1$ and ${\gamma > \alpha}$,
  \begin{equation}
    \label{eq:27}
    \Ip
    =
    \Hg
    \text{ and }
    \Sp
    =
    \emptyset.
  \end{equation}
  In the remaining case where~$\gamma \le \alpha$, the set~$\Sp$ is nonempty and the following properties hold.
  \begin{enumerate}
  \item
    The set~$\Ip$ is star-convex at~$\0$ and thus connected.
  \item
    The set~$\Sp$ is unbounded, convex, and satisfies
    \begin{equation}
      \label{eq:28}
      \Sp
      =
      \ir(\{ \varphi \in \Hg \: \delta_0 \text{ is a \textsc{Gibbs} state of~$f$ for the potential~$\varphi$} \}).
    \end{equation}
    In particular, $\Sp$ is connected.
  \item
    Each of the sets~$\Ip$ and~$\Sp$ is a regular open set.
    Furthermore, these sets are disjoint, and their union is dense in~$\Hg$.
    In particular,
    \begin{equation}
      \label{eq:29}
      \partial \Ip
      =
      \partial \Sp
      =
      \Hg \setminus (\Ip \cup \Sp).
    \end{equation}
  \end{enumerate}
\end{prop}

The first assertion of the proposition, that ${\alpha < 1}$ and ${\gamma > \alpha}$ implies ${\Ip = \Hg}$, follows from a more general result in~\cite{LiRiv14b,LiRiv14a}.
If, in addition, ${\varphi(0) = \varphi(1)}$, then see also~\cite[Proposition~11]{GarIno24} or \cite[Theorem~A]{Klo20}.
This paper gives a shorter proof specialized to intermittent maps.

The proof of \cref{p:phases} is at the end of this section.
It relies on the following direct consequence of \cite[Theorem~A.2]{0InoRiv26}, which is also used several times in the rest of the paper.

\begin{theo}
  \label{t:intermittent-potentials}
  For every~$\gamma$ in~$(0, 1]$,
  \begin{equation}
    \label{eq:30}
    \Ip
    =
    \{ \varphi \in \Hg \: P(\varphi) > \varphi(0) \},
  \end{equation}
  $\Ip$ is open in~$\Hg$, and~$P$ is real-analytic on~$\Ip$.
\end{theo}

For each~$\gamma$ in~$(0, +\infty)$, let~$\omega_\gamma$ be the function in~$\Hg$ defined by ${\omega_\gamma(x) \= -x^{\gamma}}$.
When~$\gamma$ belongs to~$(0, 1]$, this function belongs to~$\Hg$ and, for every~$\psi$ in~$\Hg$,
\begin{equation}
  \label{eq:31}
  \psi
  \ge
  \psi(0) + |\psi|_\gamma \omega_\gamma.
\end{equation}
Recall from~\cref{ss:pt-locus} that a \textsc{H{\"o}lder} continuous potential~$\varphi$ undergoes a phase transition in temperature if the function ${\beta \mapsto P(\beta \varphi)}$ fails to be real-analytic on~$(0, +\infty)$.

\begin{prop}[\textcolor{black}{\cite[Proposition~4.2]{0CorInoRiv2504}}]
  \label{p:hook}
  For every~$\gamma$ in~$(0, +\infty)$, the following properties hold.
  \begin{enumerate}
  \item
    When ${\gamma > \alpha}$, for every~$\beta$ in~$(0, +\infty)$ the inequality ${P(\beta \omega_{\gamma}) > 0}$ holds and~$\beta \omega_{\gamma}$ belongs to~$\Ip$.
  \item
    When ${\gamma \le \alpha}$, for every sufficiently large~$\beta$ in~$(0, +\infty)$ the equality ${P(\beta \omega_\gamma) = 0}$ holds and~$\beta \omega_\gamma$ is outside~$\Ip$.
  \end{enumerate}
  In particular, the potential~$\omega_\gamma$ undergoes a phase transition in temperature if and only if ${\gamma \le \alpha}$.
\end{prop}

Below is a different proof of this result.
It relies on the following notation.
For each~$j$ in~$\N$, let~$y_j$ be the unique point in~$J_0$ satisfying~${f(y_j) = x_{j - 1}}$, put ${I_j \= (y_{j + 1}, y_j]}$, and let ${m \: J_0 \to \N}$ be the function satisfying, for each~$j$ in~$\N$, the equality ${m^{-1}(j) = I_j}$.
It is \emph{the first return time to~$J_0$}.
The \emph{first return map of~$f$ to~$J_0$} is the function ${F \: J_0 \to J_0}$ defined by ${F(x) \= f^{m(x)}(x)}$.
For each~$j$ in~$\N$, the map~$F$ maps~$I_j$ diffeomorphically onto~$J_0$.

\begin{proof}[Proof of \cref{p:hook}]
  To prove item~1, put
  \begin{equation}
    \label{eq:32}
    \zeta
    \=
    \beta \sum_{j = 0}^{+ \infty} x_j^{\gamma}.
  \end{equation}
  The above sum is finite by the hypothesis ${\gamma > \alpha}$ and~\eqref{eq:18} in \cref{l:neutral-branch}.
  It follows that, for all~$k$ in~$\N_0$ and~$y$ in~$[0, x_k]$,
  \begin{equation}
    \label{eq:33}
    S_{k + 1}(\beta \omega_\gamma)(y)
    \ge
    S_{k + 1}(\beta \omega_\gamma)(x_k)
    \ge
    -\zeta.
  \end{equation}
  Recall from~\cref{ss:lematta} that~$\cP$ is the partition ${\{ [0, x_1], (x_1, 1] \}}$ of~$[0, 1]$ and that, for each~$n$ in~$\N$, its $n$-th refinement ${\bigvee_{k = 0}^{n - 1} \cP}$ is~$\cP_n$.
  Let
  \begin{equation}
    \label{eq:34}
    \iota_n \: \cP_n \to \{ 0, 1 \}^{\{ 0, \ldots, n - 1 \}}
  \end{equation}
  be the itinerary map defined in such a way that, for all ${Q \in \cP_n}$ and~$j$ in~$\{0, \ldots, n - 1 \}$, the equality ${\iota_n(Q)_j = 1}$ holds if and only if ${f^j(Q) \subseteq (x_1, 1]}$.
  The map~$\iota_n$ is a bijection.
  For each~$a_0 \ldots a_{n - 1}$ in~$\{ 0, 1 \}^{\{0, \ldots, n - 1 \}}$, put $|a_0 \ldots a_{n - 1}| \= \sum_{j = 0}^{n - 1} a_j$.
  It follows that, for every~$Q$ in~$\cP_n$,
  \begin{equation}
    \label{eq:35}
    \sup_Q S_n(\beta \omega_\gamma)
    \ge
    - (|\iota_n(Q)| + 1) \zeta.
  \end{equation}
  Hence
  \begin{equation}
    \label{eq:36}
    \begin{split}
      \sum_{Q \in \cP_n} \sup_Q \exp(S_n(\beta \omega_\gamma))
      & \ge
        \exp(- \zeta) \sum_{\underline{a} \in \{0, 1 \}^{\{0, \ldots, n - 1 \}}} \exp(- |\underline{a}| \zeta)
      \\ & =
           \exp(-\zeta) (1 + \exp(- \zeta))^n.
    \end{split}
  \end{equation}
  Taking logarithms, dividing by~$n$, taking the limit as ${n \to +\infty}$, and using~\eqref{eq:intermittent-tree} in \cref{l:pressure}, yields
  \begin{equation}
    \label{eq:37}
    P(\beta \omega_\gamma)
    \ge
    \log (1 + \exp(-\zeta)) > 0.
  \end{equation}
  Together with \cref{t:intermittent-potentials}, this implies that~$\beta \omega_\gamma$ belongs to~$\Ip$.

  To prove item~2, for all~$\beta$ in~$(0, +\infty)$ and~$n$ in~$\N$, put
  \begin{equation}
    \label{eq:38}
    \xi_n(\beta)
    \=
    \exp \left( -\beta \sum_{j = 1}^n x_j^{\gamma} \right).
  \end{equation}
  By~\eqref{eq:18} in \cref{l:neutral-branch}, the hypothesis ${\gamma \le \alpha}$ implies that there is~$\beta_0$ in~$(0, +\infty)$ such that, for every~$\beta$ in~$[\beta_0, +\infty)$,
  \begin{equation}
    \label{eq:39}
    \sum_{j = 1}^{+\infty} \xi_j(\beta)
    <
    1.
  \end{equation}
  Let~$\beta$ be in~$[\beta_0, +\infty)$.
  In view of~\eqref{eq:20} in \cref{l:pressure}, to prove ${P(\beta \omega_\gamma) = 0}$ it suffices to prove that the sum
  \begin{equation}
    \label{eq:40}
    \sum_{n = 1}^{+\infty} \sum_{x \in f^{-n}(1)} \exp (\beta S_n(\omega_\gamma)(x))
  \end{equation}
  is finite.
  For all~$j$ in~$\N$ and~$x$ in~$J_j$,
  \begin{equation}
    \label{eq:41}
    \exp (\beta S_j(\omega_\gamma)(x))
    \le
    \exp \left( -\beta \sum_{k = 1}^j x_{k + 1}^\gamma \right)
    =
    \exp(\beta x_1^\gamma) \xi_{j + 1}(\beta).
  \end{equation}
  It follows that~\eqref{eq:40} is bounded from above by
  \begin{equation}
    \label{eq:42}
    \left( 1 + \exp(\beta x_1^\gamma) \sum_{j = 1}^{+\infty} \xi_{j + 1}(\beta) \right) \left( \sum_{n = 1}^{+\infty} \sum_{\substack{x \in f^{-n}(1) \\ x \in J_0}} \exp (\beta S_n(\omega_\gamma)(x)) \right).
  \end{equation}
  The first factor above is finite by~\eqref{eq:39}.
  To prove that the second factor above is finite, note that, for all~$j$ in~$\N$ and~$y$ in~$I_j$, the point~$f(y)$ belongs to~$J_{j - 1}$, so by~\eqref{eq:41},
  \begin{equation}
    \label{eq:43}
    \exp (\beta S_j(\omega_\gamma)(y))
    \le
    \exp (-\beta y^\gamma) \exp(\beta x_1^\gamma) \xi_j(\beta)
    \le
    \xi_j(\beta).
  \end{equation}
  Let~$x$ be a point of~$f^{-n}(1)$ in~$J_0$ and~$\ell$ in~$\N$ the integer such that
  \begin{equation}
    \label{eq:44}
    m(x) + \cdots + m(F^{\ell - 1}(x))
    =
    n.
  \end{equation}
  Applying~\eqref{eq:43} repeatedly yields
  \begin{equation}
    \label{eq:45}
    \exp (\beta S_n(\omega_\gamma)(x))
    \le
    \prod_{k = 0}^{\ell - 1} \xi_{m(F^k(x)}(\beta).
  \end{equation}
  Together with~\eqref{eq:39}, this implies
  \begin{equation}
    \label{eq:46}
    \sum_{n = 1}^{+\infty} \sum_{\substack{x \in f^{-n}(1) \\ x \in J_0}} \exp (\beta S_n(\omega_\gamma)(x))
    \le
    \sum_{\ell = 1}^{+\infty} \left( \sum_{j = 1}^{+\infty} \xi_j(\beta_0) \right)^\ell
    <
    +\infty.
  \end{equation}
  This proves that the second factor in~\eqref{eq:42} is finite and completes the proof of ${P(\beta \omega_\gamma) = 0}$.
  It follows that~$\beta \omega_\gamma$ is outside~$\Ip$ by \cref{t:intermittent-potentials}.
\end{proof}

\begin{lemm}
  \label{l:phases}
  For every~$\gamma$ in~$(0, 1]$,
  \begin{equation}
    \label{eq:47}
    \Hg \setminus \Ip
    =
    \{ \varphi \in \Hg \: \delta_0 \text{ is a \textsc{Gibbs} state of~$f$ for the potential } \varphi \}.
  \end{equation}
  Furthermore, this last set is convex and its interior equals~$\Sp$.
\end{lemm}

\begin{proof}
  Putting ${\hSp \= \Hg \setminus \Ip}$, \cref{t:intermittent-potentials} implies
  \begin{multline}
    \label{eq:48}
    \hSp
    =
    \{ \varphi \in \Hg \: P(\varphi) = \varphi(0) \}
    \\ =
    \{ \varphi \in \Hg \: \delta_0 \text{ is a \textsc{Gibbs} state of~$f$ for the potential } \varphi \}.
  \end{multline}
  In particular, ${\Sp \subseteq \ir(\hSp)}$.
  It also follows that~$\hSp$ is convex, because~$P$ is convex and ${\varphi \mapsto \varphi(0)}$ is linear.

  To complete the proof of the last assertion, let~$\varphi$ be in $\ir(\hSp)$ and let~$\beta_*$ in~$(0, 1)$ be sufficiently close to~$1$ so that~$\beta_* \varphi$ belongs to~$\hSp$.
  Then, ${P(\beta_* \varphi) = \beta_* \varphi(0)}$ and by \cref{c:states} with ${\beta = 1}$, the measure~$\delta_0$ is the unique \textsc{Gibbs} state of~$f$ for the potential~$\varphi$.
  Since this holds for every potential in~$\ir(\hSp)$ and this set is open, ${\ir(\hSp) \subseteq \Sp}$.
  The proof of the lemma is thus complete.
\end{proof}

\begin{lemm}
  \label{l:stationarization}
  For all~$\gamma$ in~$(0, \min \{ \alpha, 1\}]$, $\psi_0$ in ${\Hg \setminus \Ip}$, and~$\tau$ in~$(0, +\infty)$, the potential ${\psi_0 + \tau \omega_\gamma}$ belongs to~$\Sp$.
\end{lemm}

\begin{proof}
  Let~$\Delta$ in~$\Hg$ be such that ${|\Delta|_{\gamma} \le \tau}$.
  By~\eqref{eq:31} with ${\psi = -\Delta}$,
  \begin{equation}
    \label{eq:49}
    \psi_0 + \tau \omega_{\gamma} + \Delta
    \le
    \psi_0 + \Delta(0).
  \end{equation}
  Together with the hypothesis that~$\psi_0$ is outside~$\Ip$ and \cref{t:intermittent-potentials}, this implies
  \begin{multline}
    \label{eq:50}
    P(\psi_0 + \tau \omega_{\gamma} + \Delta)
    \le
    P(\psi_0 + \Delta(0))
    =
    P(\psi_0) + \Delta(0)
    =
    (\psi_0 + \Delta)(0)
    \\ =
    h_{\delta_0} + \il \psi_0 + \Delta \dd \delta_0
    \le
    P(\psi_0 + \tau \omega_{\gamma} + \Delta).
  \end{multline}
  Hence~$\delta_0$ is a \textsc{Gibbs} state for the potential ${\psi_0 + \tau \omega_{\gamma} + \Delta}$.
  By \cref{l:phases}, $\psi_0 + \tau \omega_{\gamma}$ belongs to~$\Sp$.
\end{proof}

\begin{proof}[Proof of \cref{p:phases}]
  To prove that~$\Ip$ contains~$\0$, let~$\nu_0$ in~$\sM$ be such that ${h_{\nu_0} > 0}$ (see, for example, \cref{l:mme}).
  So,
  \begin{equation}
    \label{eq:51}
    \0(0)
    =
    0
    <
    h_{\nu_0} + \il \0 \dd \nu_0
    \le
    P(\0).
  \end{equation}
  Together with \cref{t:intermittent-potentials}, this implies that~$\0$ belongs to~$\Ip$.

  To prove~\eqref{eq:27}, suppose ${\alpha < 1}$ and ${\gamma > \alpha}$, and let~$\varphi$ be in~$\Hg$.
  By~\eqref{eq:31} with ${\psi = \varphi}$ and \cref{p:hook}(1) with ${\beta = | \varphi |_\gamma}$,
  \begin{equation}
    \label{eq:52}
    \varphi - \varphi(0)
    \ge
    | \varphi |_\gamma \omega_\gamma
    \text{ and }
    P(\varphi) - \varphi(0)
    =
    P(\varphi - \varphi(0))
    \ge
    P(| \varphi |_\gamma \omega_\gamma)
    >
    0.
  \end{equation}
  Together with \cref{t:intermittent-potentials}, this implies that~$\varphi$ belongs to~$\Ip$.

  Henceforth, suppose ${\gamma \le \alpha}$.
  To prove item~1, recall that~$\Ip$ contains~$\0$ and let~$\varphi$ be in~$\Ip$.
  By \cref{t:intermittent-potentials}, for every~$\beta$ in~$(0, 1)$,
  \begin{equation}
    \label{eq:53}
    P(\beta \varphi)
    =
    \sup_{\nu \in \sM} \left\{ h_\nu + \beta \il \varphi \dd \nu \right\}
    \ge
    \beta P(\varphi)
    >
    \beta \varphi(0).
  \end{equation}
  Using \cref{t:intermittent-potentials} again, yields that~$\beta \varphi$ belongs to~$\Ip$.

  For item~2, by \cref{p:hook}(2) there is~$\beta_*$ in~$(0, +\infty)$ such that~$\beta_* \omega_\gamma$ is outside~$\Ip$.
  Combined with \cref{l:stationarization} with ${\psi_0 = \beta_* \omega_\gamma}$, this implies that~$\Sp$ is unbounded and thus nonempty.
  That~$\Sp$ is convex, and the last assertion of item~2 both follow from \cref{l:phases}.

  For item~3, $\Sp$ is open by definition, $\Ip$ is open by \cref{t:intermittent-potentials}, and~$\Ip$ and~$\Sp$ are disjoint by definition.
  Moreover, ${\Sp = \ir(\Hg \setminus \Ip)}$ by \cref{l:phases}.
  This implies that~$\Sp$ is a regular open set,
  \begin{equation}
    \label{eq:54}
    \partial \Ip
    =
    \Hg \setminus (\Ip \cup \Sp),
    \text{ and }
    \cl(\Ip \cup \Sp)
    =
    \Hg.
  \end{equation}
  The proof of the remaining assertions relies on the convexity of ${\Hg \setminus \Ip}$ by \cref{l:phases} and on ${\ir(\Sp) \neq \emptyset}$.
  This implies that ${\Hg \setminus \Ip}$ equals the closure of its interior and, as a consequence, the open set~$\Ip$ is regular and ${\partial \Ip = \partial \Sp}$.
\end{proof}

\begin{proof}[Proof of \cref{t:pt-locus'}]
  Put ${\cS \= \Sp\cup\PTg}$, so ${\cS = \Hg\setminus\Ip}$.
  By Proposition~\ref{p:phases}(2,3), the set~$\cS$ is closed and convex, ${\ir(\cS) = \Sp}$, and ${\partial \cS = \PTg}$.
  Moreover, put
  \begin{equation}
    \label{eq:55}
    \Xi
    \=
    (\Hg/V) \setminus \partial\pi(\Sp).
  \end{equation}
  Since the quotient map~$\pi$ is open, $\pi(\Sp)$ is open and thus~$\Xi$ is open and dense in~$\Hg/V$.
  Moreover, ${\cS = \cl(\Sp)}$ and the continuity of~$\pi$ imply
  \begin{equation}
    \label{eq:image-C-in-closure}
    \pi(\cS)
    \subseteq
    \cl(\pi(\Sp)).
  \end{equation}

  Let~$\xi$ be in~$\Xi$ and put
  \begin{equation}
    \label{eq:56}
    \cS(\xi)
    \=
    V_\xi \cap \Sp,
    \cI(\xi)
    \=
    V_\xi \cap \Ip,
    \text{ and }
    \PT{\xi}
    \=
    V_\xi \cap \PTg.
  \end{equation}
  Suppose first that~$\xi$ is outside~$\pi(\Sp)$.
  Since~$\xi$ is outside~$\partial\pi(\Sp)$ by hypothesis, it is also outside~$\cl(\pi(\Sp))$.
  It follows from~\cref{eq:image-C-in-closure} that~$\xi$ is outside~$\pi(\cS)$.
  Thus, ${V_\xi\cap\cS = \emptyset}$ and ${V_\xi \subseteq \Ip}$.
  In particular, ${\HD(\PT{\xi}) = 0}$.

  Suppose~$\xi$ belongs to~$\pi(\Sp)$, so~$\cS(\xi)$ is nonempty, open, and convex.
  It follows that~$\cS(\xi)$ is regular, see, for example, \cite[Theorem~6.3]{Roc97}.
  Denote by~$\partial_{\xi}$, ${\ir_{\xi}}$, and~$\cl_{\xi}$ the relative boundary, interior, and closure in~$V_{\xi}$, respectively.
  The hypothesis ${V \not\subseteq \Ns}$ implies that there are~$v$ in ${V \setminus \Ns}$ and~$\nu$ in~$\sM$ satisfying ${\il v \dd \nu \neq v(0)}$.
  Replacing~$v$ by~$-v$ if necessary, suppose ${\il v \dd \nu > v(0)}$.
  Put
  \begin{equation}
    \label{eq:57}
    a
    \=
    \il v \dd \nu - v(0),
  \end{equation}
  so ${a > 0}$, and let~$\varphi_{\xi}$ in~$V_{\xi}$ be given.
  Then, for every~$t$ in~$\R$,
  \begin{equation}
    \label{eq:58}
    P(\varphi_{\xi} + tv) - (\varphi_{\xi} + tv)(0)
    \ge
    h_{\nu} + \il \varphi_{\xi} + tv \dd \nu - (\varphi_{\xi} + tv)(0)
    =
    h_{\nu} + \il \varphi_{\xi} \dd \nu - \varphi_{\xi}(0) + ta.
  \end{equation}
  So, for every sufficiently large~$t$ the inequality ${P(\varphi_{\xi} + tv) > (\varphi_{\xi} + tv)(0)}$ holds, and~$\varphi_{\xi} + tv$ belongs to~$\cI (\xi)$ by \cref{t:intermittent-potentials}.
  This proves ${\cI(\xi) \neq \emptyset}$ and completes the proof of~\eqref{eq:9}.
  The next step is to prove
  \begin{equation}
    \label{eq:section-PT-boundary}
    \PT{\xi}
    =
    \partial_{\xi}\cS(\xi).
  \end{equation}
  Since~$\cS(\xi)$ is an open subset of~$V_\xi$,
  \begin{equation}
    \label{eq:59}
    \partial_\xi\cS(\xi)
    =
    \cl_\xi(\cS(\xi)) \setminus \cS(\xi)
    \text{ and }
    \cl_\xi(\cS(\xi))
    \subseteq
    V_\xi \cap \cl(\Sp).
  \end{equation}
  Combined with ${\cl(\Sp) \setminus \Sp = \PTg}$, this yields ${\partial_\xi\cS(\xi) \subseteq \PT{\xi}}$.
  For the direct inclusion, let~$\psi_0$ be in~$\PT{\xi}$ and~$\psi$ in~$\cS(\xi)$.
  Since~$\cS$ is convex and~$\psi$ belongs to its interior, for every~$t$ in~$(0,1)$,
  \begin{equation}
    \label{eq:60}
    (1-t)\psi_0 + t \Sp
    \subseteq
    \cS.
  \end{equation}
  Hence, $(1-t)\psi_0+t\psi$ belongs to~$\ir(\cS)$, which equals~$\Sp$.
  Since~$\psi_0$ and~$\psi$ are both in~$V_\xi$, it follows that $(1-t)\psi_0+t\psi$ belongs to~$\cS(\xi)$.
  Letting~$t$ tend to~$0$ yields that~$\psi_0$ belongs to~$\cl_\xi(\cS(\xi))$.
  Since~$\psi_0$ is outside~$\cS(\xi)$, the first equality in~\eqref{eq:59} yields that~$\psi_0$ belongs to~$\partial_\xi \cS(\xi)$.
  This proves the direct inclusion in~\eqref{eq:section-PT-boundary}.
  Combining~\eqref{eq:9} and~\eqref{eq:section-PT-boundary} yields ${\PT{\xi} \neq\emptyset}$.
  Let~$\varphi_0$ in~$\PT{\xi}$ and~$r$ in~$(0, +\infty)$ be given and put
  \begin{equation}
    \label{eq:61}
    \cQ
    \=
    \{ \varphi \in \cl_{\xi}(\cS(\xi)) \: \| \varphi - \varphi_0 \|_{\gamma} \le r \}.
  \end{equation}
  The set~$\cQ$ is compact and convex.
  Since~$\cS(\xi)$ is open and~$\varphi_0$ belongs to~$\partial_{\xi} \cS(\xi)$ by~\eqref{eq:section-PT-boundary}, the set~$\ir_{\xi}(Q)$ is nonempty.
  It follows that, on a neighborhood of~$\varphi_0$, the set~$\partial_{\xi} Q$ is bi-\textsc{Lipschitz} equivalent to an open subset of~$\R^{k-1}$, see, for example, \cite[Section~13.4.5]{MakPod13}.
  Since by~\eqref{eq:section-PT-boundary} the sets~$\partial_{\xi} Q$ and~$\PT{\xi}$ agree on
  \begin{equation}
    \label{eq:62}
    \{ \varphi \in V_{\xi} \: \| \varphi - \varphi_0 \|_{\gamma} < r \},
  \end{equation}
  it follows that, on a neighborhood of~$\varphi_0$, the set~$\PT{\xi}$ is bi-\textsc{Lipschitz} equivalent to an open subset of~$\R^{k-1}$.
  Hence, ${\PT{\xi}}$ is a \textsc{Lipschitz} manifold of dimension~${k-1}$ and ${\HD(\PT{\xi}) = k - 1}$.
\end{proof}

\subsection{Receding from zero temperature}
\label{ss:receding-to-zero-temperature}
This section relates the geometry at infinity of the stationary phase and the phase transition locus to the set of potentials for which~$\delta_0$ is a ground state of~$f$.
This is summarized in \cref{p:receding-to-zero-temperature} below.

Let~$V$ be a real vector space.
A subset~$C$ of~$V$ is a \emph{cone}, if for all~$v$ in~$C$ and~$\beta$ in~$(0, +\infty)$, the vector~$\beta v$ belongs to~$C$.
In this case, $C$ is \emph{convex} if, for all~$v$ and~$v'$ in~$C$, the vector ${v + v'}$ belongs to~$C$ and it is \emph{salient} if ${C \cap -C}$ is reduced to the zero vector of~$V$.
The \emph{recession cone} of a subset~$S$ of~$V$ is the convex cone~$\rec(S)$ defined by
\begin{equation}
  \label{eq:63}
  \rec(S)
  \=
  \{ v \in V \: \text{ for all~$s$ in~$S$ and~$t$ in~$(0, +\infty)$, } s + t v \in S \}.
\end{equation}
Moreover, the \emph{lineality space of~$S$} is the vector subspace ${\rec(S) \cap -\rec(S)}$ of~$V$.

Recall from~\eqref{eq:7} that, for each~$\gamma$ in~$(0, 1]$,
\begin{equation}
  \label{eq:64}
  \Ns
  =
  \left\{ \varphi \in \Hg \: \text{ for every~$\nu$ in~$\sM$, } \il \varphi \dd \nu = \varphi(0) \right\},
\end{equation}
and, from~\eqref{eq:11}, that~$\Sc$ is the set of potentials in~$\Hg$ for which~$\delta_0$ is a ground state of~$f$.

\begin{prop}
  \label{p:receding-to-zero-temperature}
  For every~$\gamma$ in~$(0, \min \{ \alpha, 1\}]$, the following properties hold.
  \begin{enumerate}
  \item
    The following chain of equalities holds,
    \begin{equation}
      \label{eq:65}
      \rec(\Sp)
      =
      \rec(\Sp \cup \PTg)
      =
      \Sc.
    \end{equation}
  \item
    The lineality spaces of~$\Sp$ and~${\Sp \cup \PTg}$ are both equal to~$\Ns$.
    Moreover, this vector subspace of~$\Hg$ is of infinite co-dimension.
  \end{enumerate}
\end{prop}

\begin{proof}
  The proof of item~1 begins by showing
  \begin{equation}
    \label{eq:66}
    \Sc \subseteq \rec(\Sp).
  \end{equation}
  Let~$\varphi$ be in~$\Sc$, $\varphi_0$ in~$\Sp$, and~$\nu_0$ a \textsc{Gibbs} state of~$f$ for the potential ${\varphi_0 + \varphi}$.
  Then,
  \begin{equation}
    \label{eq:67}
    P(\varphi_0 + \varphi)
    =
    h_{\nu_0} + \il \varphi_0 + \varphi \dd \nu_0
    \le
    P(\varphi_0) + \il \varphi \dd \nu_0
    \le
    P(\varphi_0) + \varphi(0)
    =
    (\varphi_0 + \varphi)(0)
  \end{equation}
  and thus~${\varphi_0 + \varphi}$ is outside~$\Ip$ by \cref{t:intermittent-potentials}.
  Since this holds for every~$\varphi_0$ in~$\Sp$,
  \begin{equation}
    \label{eq:68}
    \Sp + \varphi
    \subseteq
    \Hg \setminus \Ip.
  \end{equation}
  The openness of~$\Sp$ and \cref{p:phases}(3) yield
  \begin{equation}
    \label{eq:69}
    \Sp + \varphi
    \subseteq
    \ir(\Hg \setminus \Ip)
    =
    \Sp.
  \end{equation}
  Since~$\Sp$ is convex by \cref{p:phases}(2), this yields that~$\varphi$ belongs to~$\rec(\Sp)$ and proves~\eqref{eq:66}.

  The inclusion ${\rec(\Sp) \subseteq \rec(\Sp \cup \PTg)}$ being trivial, to complete the proof of item~1 it suffices to show
  \begin{equation}
    \label{eq:70}
    \rec(\Sp \cup \PTg)
    \subseteq
    \Sc.
  \end{equation}
  To do this, let~$\varphi$ be in~$\rec(\Sp \cup \PTg)$, $\nu_{\dag}$ in~$\sM$ such that
  \begin{equation}
    \label{eq:71}
    \il \varphi \dd \nu_{\dag}
    =
    \sup_{\nu \in \sM} \il \varphi \dd \nu,
  \end{equation}
  and~$\varphi_\S$ in~$\Sp$.
  Then, for every~$t$ in~$(0, +\infty)$, the potential ${\varphi_\S + t \varphi}$ belongs to ${\Sp \cup \PTg}$.
  Together with \cref{p:phases}(3) and \cref{t:intermittent-potentials}, this implies
  \begin{equation}
    \label{eq:72}
    (\varphi_\S + t \varphi)(0)
    =
    P(\varphi_\S + t \varphi)
    \ge
    \il \varphi_\S + t \varphi \dd \nu_{\dag}
    =
    \il \varphi_\S \dd \nu_{\dag} + t \sup_{\nu \in \sM} \il \varphi \dd \nu.
  \end{equation}
  Dividing by~$t$ and rearranging yields
  \begin{equation}
    \label{eq:73}
    \varphi(0)
    \ge
    \sup_{\nu \in \sM} \il \varphi \dd \nu - \frac{1}{t} \varphi_\S(0) + \frac{1}{t} \il \varphi_\S \dd \nu_{\dag}
    \ge
    \sup_{\nu \in \sM} \il \varphi \dd \nu - \frac{2}{t} \left\| \varphi_\S \right\|.
  \end{equation}
  Taking the limit as ${t \to +\infty}$ yields that~$\varphi$ belongs to~$\Sc$.
  This completes the proof of~\eqref{eq:70}, and thus of~\eqref{eq:65} and item~1.

  To prove item~2, note first that, for every~$\varphi$ in ${\Sc \cap - \Sc}$,
  \begin{equation}
    \label{eq:74}
    \sup_{\nu \in \sM} \il \varphi \dd \nu
    =
    \varphi(0)
    =
    \inf_{\nu \in \sM} \il \varphi \dd \nu,
  \end{equation}
  so, for every~$\nu$ in~$\sM$, the equality ${\il \varphi \dd \nu = \varphi(0)}$ holds and thus~$\varphi$ belongs to~$\Ns$.
  Conversely, every~$\varphi$ in~$\Ns$ satisfies~\eqref{eq:74} and is therefore in ${\Sc \cap -\Sc}$.

  To prove the last statement of item~2, denote by~$\sM'$ the infinite subset of~$\sM$ of those measures supported on a single periodic orbit of~$f$.
  For each~$\nu$ in~$\sM'$, let~$\ell_{\nu}$ be the linear functional of~$\Hg$ defined by
  \begin{equation}
    \label{eq:75}
    \ell_{\nu}(\varphi)
    \=
    \varphi(0) - \il \varphi \dd \nu.
  \end{equation}
  Then,
  \begin{equation}
    \label{eq:76}
    \Ns
    \subseteq
    \bigcap_{\nu \in \sM'} \Ker(\ell_{\nu}).
  \end{equation}
  Furthermore, the set of linear functionals ${\{ \ell_{\nu} \: \nu \in \sM' \}}$ is linearly independent, because the supports of the measures in~$\sM'$ are pairwise disjoint.
  This implies that the right side of~\eqref{eq:76} has infinite co-dimension in~$\Hg$ and thus that the same holds for~$\Ns$.
\end{proof}

\subsection{Phase transition locus}
\label{ss:proof-pt-locus}

This section establishes fundamental properties of the phase transition locus that are summarized in the following proposition, and derives \cref{t:pt-locus}.

\begin{prop}
  \label{p:pt-locus}
  For every~$\gamma$ in~$(0, \min \{ \alpha, 1 \}]$, the following properties hold.
  \begin{enumerate}
  \item
    The set~$\PTg$ is linearly homeomorphic to the graph of a real-valued convex function.
    More precisely, there is a co-dimension~1 vector subspace~$\cK$ of~$\Hg$, a convex function ${\tau \: \cK \to (0, +\infty)}$, and a linear homeomorphism from~$\Hg$ onto~$\cK \times \R$ that is the identity on~$\cK$ and that maps~$\PTg$ onto the graph of~$\tau$.
  \item
    Every affine space contained in ${\PTg \cup \Sp}$ is a translate of a vector subspace of~$\Ns$ and it is of infinite co-dimension in~$\Hg$.
  \end{enumerate}
\end{prop}

For every~$\gamma$ in~$(0, \min \{ \alpha, 1 \}]$, the set~$\PTg$ contains a translate of~$\Ns$ and the phase diagram for~$f$ descends to the quotient space~$\Hg / \Ns$ by~\eqref{eq:8}.
\cref{p:pt-locus}(2) implies that the (convex and unbounded) set
\begin{equation}
  \label{eq:77}
  (\PTg \cup \Sp) / \Ns
\end{equation}
contains no nontrivial affine subspace of ${\Hg / \Ns}$.
Equivalently, the recession cone of this set is salient.

The proofs of \cref{p:pt-locus} and \cref{t:pt-locus} rely on the following lemma.

\begin{lemm}
  \label{l:pt-locus}
  Let~$\gamma$ be in~$(0, \min \{ \alpha, 1\}]$ and~$\varphi_0$ in~$\Hg$.
  Then, there is~$\tau_0$ in~$\R$ such that, for every~$\tau$ in~$\R$,
  \begin{equation}
    \label{eq:78}
    \varphi_0 + \tau \omega_\gamma
    \in
    \begin{cases}
      \Ip
      & \text{if } \tau < \tau_0;
      \\
      \PTg
      & \text{if } \tau = \tau_0;
      \\
      \Sp
      & \text{if } \tau > \tau_0.
    \end{cases}
  \end{equation}
\end{lemm}

\begin{proof}
  For each~$\tau$ in~$\R$, put ${\psi_\tau \= \varphi_0 + \tau \omega_\gamma}$ and let ${g \: \R \to [0, +\infty)}$ be defined by
  \begin{equation}
    \label{eq:79}
    g(\tau)
    \=
    P(\psi_\tau) - \psi_\tau(0).
  \end{equation}
  The inequalities ${\omega_\gamma \le 0}$ and ${\omega_\gamma(0) = 0}$ yield, for all~$\tau$ and~$\tau'$ in~$\R$ satisfying ${\tau \le \tau'}$,
  \begin{equation}
    \label{eq:80}
    \psi_\tau
    \ge
    \psi_{\tau'}
    \text{ and }
    g(\tau)
    =
    P(\psi_\tau) - \psi_\tau(0)
    \ge
    P(\psi_{\tau'}) - \psi_\tau(0)
    =
    g(\tau').
  \end{equation}
  In particular, $g$ is nonincreasing.
  Moreover, since~$\delta_1$ belongs to~$\sM$, for every sufficiently negative~$\tau$ in~$\R$,
  \begin{equation}
    \label{eq:81}
    g(\tau)
    \ge
    \il \psi_\tau \dd \delta_1 - \psi_\tau(0)
    =
    \varphi_0(1) - \tau - \varphi_0(0)
    >
    0.
  \end{equation}
  The next step is to show that~$g$ attains the value~$0$.
  To do this, let~$\beta_0$ in~$(0, +\infty)$ be such that ${P(\beta_0 \omega_\gamma) = 0}$ (\cref{p:hook}(2)) and put ${\tau_\dag \= \beta_0 + |\varphi_0|_\gamma}$.
  By~\eqref{eq:31} with ${\psi = - \varphi_0}$,
  \begin{equation}
    \label{eq:82}
    \psi_{\tau_\dag}
    =
    \varphi_0 + (\beta_0 + |\varphi_0|_\gamma) \omega_\gamma
    \le
    \beta_0 \omega_\gamma + \varphi_0(0)
  \end{equation}
  and thus
  \begin{equation}
    \label{eq:83}
    P(\psi_{\tau_\dag})
    \le
    P(\beta_0 \omega_\gamma + \varphi_0(0))
    =
    P(\beta_0 \omega_\gamma) + \varphi_0(0)
    =
    \varphi_0(0)
    =
    h_{\delta_0} + \il \psi_{\tau_\dag} \dd \delta_0
    \le
    P(\psi_{\tau_\dag}).
  \end{equation}
  This implies ${g(\tau_\dag) = 0}$, as wanted.
  It follows that there is~$\tau_0$ in~$\R$ such that~$g$ is strictly positive on~$(-\infty, \tau_0)$ and vanishes identically on~$[\tau_0, +\infty)$.

  Combining \cref{t:intermittent-potentials} with the definition of~$\PTg$ yields that, for every~$\tau$ in~$(-\infty, \tau_0)$, the potential~$\psi_\tau$ belongs to~$\Ip$ and that, for ${\tau = \tau_0}$, the potential~$\psi_{\tau_0}$ belongs to~$\PTg$.
  \cref{l:stationarization} with ${\psi_0 = \psi_{\tau_0}}$ implies that, for every~$\tau$ in~$(\tau_0, +\infty)$, the potential~${\psi_\tau}$ belongs to~$\Sp$.
\end{proof}

\begin{proof}[Proof of \cref{p:pt-locus}]
  Fix~$\nu$ in ${\sM \setminus \{ \delta_0 \}}$ and let ${\ell \: \Hg \to \R}$ be the continuous linear functional defined by
  \begin{equation}
    \label{eq:84}
    \ell(\psi)
    \=
    \psi(0) - \il \psi \dd \nu.
  \end{equation}
  The hypothesis ${\nu \neq \delta_0}$ implies
  \begin{equation}
    \label{eq:85}
    \ell(\omega_\gamma)
    =
    \il x^{\gamma} \dd \nu(x)
    >
    0.
  \end{equation}
  Define the linear projection ${P \: \Hg \to \Ker(\ell)}$, by
  \begin{equation}
    \label{eq:86}
    P(\varphi)
    \=
    \varphi - \frac{\ell(\varphi)}{\ell(\omega_\gamma)} \omega_\gamma
  \end{equation}
  and a linear map ${\Hg \to \Ker(\ell) \times \R}$ by
  \begin{equation}
    \label{eq:87}
    \varphi \mapsto \left( P(\varphi), \frac{\ell(\varphi)}{\ell(\omega_\gamma)} \right).
  \end{equation}
  This last map is a homeomorphism with inverse ${(\psi, t) \mapsto \psi + t \omega_\gamma}$.

  For each~$\varphi$ in~$\Ker(\ell)$, denote by~$\tau(\varphi)$ the real number~$\tau_0$ given by \cref{l:pt-locus} with ${\varphi_0 = \varphi}$.
  This lemma implies that the homeomorphism~\eqref{eq:87} maps ${\PTg \cup \Sp}$ onto the epigraph of the function ${\tau \: \Ker(\ell) \to \R}$ so defined.
  Since ${\PTg \cup \Sp}$ is convex by \cref{p:phases}(2,~3), it follows that the function~$\tau$ is convex and thus continuous.
  By construction, the homeomorphism~\eqref{eq:87} maps~$\PTg$ onto the graph of~$\tau$.
  Hence, to complete the proof of the desired assertion with ${\cK = \Ker(\ell)}$, it only remains to show that the function~$\tau$ takes values in~$(0, +\infty)$.
  To do this, let~$\varphi$ be in~$\Ker(\ell)$ and recall that ${\varphi + \tau(\varphi) \omega_\gamma}$ belongs to~$\PTg$ by the definition of~$\tau(\varphi)$.
  By \cref{t:intermittent-potentials}, the hypothesis that~$\varphi$ belongs to~$\Ker(\ell)$, and the \cref{l:key} in~\cref{ss:lematta},
  \begin{equation}
    \label{eq:88}
    P(\varphi + \tau(\varphi) \omega_\gamma)
    =
    (\varphi + \tau(\varphi) \omega_\gamma)(0)
    =
    \varphi(0)
    =
    \il \varphi \dd \nu
    <
    P(\varphi).
  \end{equation}
  Together with ${\omega_\gamma \le 0}$, this implies ${\tau(\varphi) > 0}$.

  To prove item~2, let~$\cA$ be an affine subspace of~$\Hg$ contained in ${\PTg \cup \Sp}$.
  The vector subspace of~$\Hg$ parallel to~$\cA$ is contained in the lineality space of~$\Sc$, which equals~$\Ns$ and has infinite co-dimension in~$\Hg$ by \cref{p:receding-to-zero-temperature}(2).
  It follows that~$\cA$ also has infinite co-dimension in~$\Hg$.
\end{proof}

The proof of \cref{t:pt-locus} also relies on the following lemma.

\begin{lemm}
  \label{l:singular-family}
  Let~$\gamma$ be in~$(0, \min \{ \alpha, 1 \}]$, $I$ an open interval of~$\R$, and~$(\varphi_\tau)_{\tau \in I}$ a real-analytic family of potentials in~$\Hg$.
  Suppose that there is a parameter~$\tau_0$ in~$I$, such that, for every~$\tau$ in~$I$, the potential~$\varphi_\tau$ belongs to~$\Ip$ if ${\tau < \tau_0}$ and it is outside~$\Ip$ if ${\tau \ge \tau_0}$.
  Then, ${\tau \mapsto P(\varphi_\tau)}$ fails to be real-analytic at~$\tau_0$.
\end{lemm}

\begin{proof}
  Let ${p \: I \to [0, +\infty)}$ be defined by
  \begin{equation}
    \label{eq:89}
    p(\tau)
    \=
    P(\varphi_\tau) - \varphi_\tau(0).
  \end{equation}
  By \cref{t:intermittent-potentials}, the hypotheses imply that~$p$ is strictly positive on ${I \cap (-\infty, \tau_0)}$ and that it vanishes identically on ${I \cap [\tau_0, +\infty)}$.
  In particular, $p$ fails to be real-analytic at~$\tau_0$.
  Since the function ${\varphi \mapsto \varphi(0)}$ is real-analytic on~$\Hg$, it follows that ${\tau \mapsto P(\varphi_\tau)}$ also fails to be real-analytic at~$\tau_0$.
\end{proof}

\begin{proof}[Proof of \cref{t:pt-locus}]
  The function~$P$ is real-analytic on~$\Ip$ (\cref{t:intermittent-potentials}) and it coincides with the real-analytic function ${\varphi \mapsto \varphi(0)}$ on~$\Sp$.
  So, to prove the first assertion, it suffices to show that, for every~$\varphi_0$ in~$\PTg$, the function~$P$ fails to be real analytic at~$\varphi_0$.
  To do this, for each~$\tau$ in~$\R$ put ${\varphi_{\tau} \= \varphi_0 + \tau \omega_{\gamma}}$.
  By \cref{l:pt-locus} there is~$\tau_0$ in~$\R$ such that, for every~$\tau$ in~$\R$, the potential ${\varphi_0 + \tau \omega_{\gamma}}$ belongs to~$\Ip$ if ${\tau < \tau_0}$ and it is outside~$\Ip$ if ${\tau \ge \tau_0}$.
  Then, ${\tau \mapsto P(\varphi_\tau)}$ fails to be real-analytic at~$\tau_0$ by \cref{l:singular-family} and thus~$P$ fails to be real-analytic at~$\varphi_0$.

  \cref{p:phases} implies that ${\PTg \neq \emptyset}$ is equivalent to~$\gamma$ belonging to~$(0, \min\{\alpha, 1\}]$.
  To prove the second assertion, suppose these properties hold.
  Let~$\cK$ and~$\tau$ be given by \cref{p:pt-locus}(1) and denote by
  \begin{equation}
    \label{eq:90}
    L \: \Hg \to \cK \times \R
  \end{equation}
  the linear homeomorphism given by this result.
  The map ${H \: \cK \times \R \to \cK \times \R}$ defined by
  \begin{equation}
    \label{eq:91}
    H(\varphi_0, h)
    \=
    (\varphi_0, h - \tau(\varphi_0))
  \end{equation}
  is a homeomorphism mapping the graph of~$\tau$ onto~$\cK \times \{ 0 \}$.
  It follows from \cref{p:pt-locus}(1) that ${L^{-1} \circ H \circ L}$ is a homeomorphism from~$\Hg$ onto itself mapping~$\PTg$ onto~$\cK$.
  In particular, $\PTg$ is a topological submanifold of co-dimension~1 of~$\Hg$.
  If~$\PTg$ were an affine subspace of~$\Hg$, then it would be of co-dimension~1, which is absurd by \cref{p:pt-locus}(2).
\end{proof}

\section{On the persistence of phase transitions in temperature}
\label{s:pt}

This section investigates the persistence of phase transitions in temperature.
\cref{ss:zero-to-positive} relates, for each~$\gamma$ in~$(0, \min\{\alpha, 1\}]$, the set~$\Sp$ to the cone~$\Sc$ (\cref{p:zero-to-positive}) and proves \cref{t:persistence-locus}.
\cref{ss:nonpersistent-pt} states and proves the main results when ${\gamma < \alpha}$ (\cref{p:nonpersistent-pt}).
\cref{ss:persistent-pt} gives several characterizations of the potentials that undergo a phase transition in temperature in the remaining case, where ${\alpha \le 1}$ and ${\gamma = \alpha}$ (\cref{p:persistent-pt} in~\cref{ss:persistent-pt}).
\cref{p:persistent-pt} is the main ingredient in the proof of \cref{t:alpha-rigidity}, which is given in~\cref{ss:proof-alpha-rigidity}.

\subsection{From zero to low temperatures}
\label{ss:zero-to-positive}

This section proves \cref{t:persistence-locus}.
The main new ingredient is the following proposition.

\begin{prop}
  \label{p:zero-to-positive}
  For every~$\gamma$ in~$(0, \min\{\alpha, 1\}]$,
  \begin{equation}
    \label{eq:92}
    \ir(\Sc)
    =
    \bigcup_{\beta \in (0, +\infty)} \beta \Sp.
  \end{equation}
  In particular,
  \begin{equation}
    \label{eq:93}
    \Sp
    \subseteq
    \ir(\Sc)
    \text{ and }
    \PTg
    \subseteq
    \cl(\Sp)
    \subseteq
    \Sc.
  \end{equation}
\end{prop}

\begin{proof}
  For all~$\beta$ in~$(0, +\infty)$ and~$\varphi$ in~$\beta \Sp$,
  \begin{equation}
    \label{eq:94}
    \beta^{-1} \varphi(0)
    =
    P \left( \beta^{-1} \varphi \right)
    \ge
    \sup_{\nu \in \sM} \il \beta^{-1} \varphi \dd \nu
    \ge
    \beta^{-1} \varphi(0).
  \end{equation}
  This implies that~$\varphi$ belongs to~$\Sc$ and ${\beta \Sp \subseteq \Sc}$.
  Since~$\Sp$ is open and this holds for every~$\beta$ in~$(0, +\infty)$,
  \begin{equation}
    \label{eq:95}
    \bigcup_{\beta \in (0, +\infty)} \beta \Sp
    \subseteq
    \ir(\Sc).
  \end{equation}

  To prove the reverse inclusion, let~$\varphi$ be in~$\ir(\Sc)$ and~$\varphi_\bullet$ in~$\Sp$.
  If~$\varphi$ and~$\varphi_\bullet$ are linearly dependent, then there is~$\beta$ in~$(0, +\infty)$ such that~$\varphi$ belongs to~$\beta \Sp$.
  Suppose~$\varphi$ and~$\varphi_\bullet$ are linearly independent, denote by~$\cP$ the plane in~$\Hg$ generated by~$\varphi$ and~$\varphi_\bullet$, and put
  \begin{equation}
    \label{eq:96}
    \cG
    \=
    \Sc \cap \cP.
  \end{equation}
  The set~$\cG$ is a closed convex cone and thus there are linear functionals~$\ell$ and~$\ell'$ of~$\cP$, so that
  \begin{equation}
    \label{eq:97}
    \cG
    =
    \{ \psi \in \cP \: \ell(\psi) \ge 0, \ell'(\psi) \ge 0 \}.
  \end{equation}
  Since~$\varphi$ belongs to~$\ir(\Sc)$, the inequalities ${\ell(\varphi) > 0}$ and ${\ell'(\varphi) > 0}$ hold.
  Let~$t$ in~$(0, +\infty)$ be sufficiently large so that
  \begin{equation}
    \label{eq:98}
    t \ell(\varphi)
    \ge
    \ell(\varphi_\bullet)
    \text{ and }
    t \ell'(\varphi)
    \ge
    \ell'(\varphi_\bullet),
  \end{equation}
  and put ${\hvarphi \= t\varphi - \varphi_\bullet}$.
  Then, ${\ell(\hvarphi) \ge 0}$, ${\ell'(\hvarphi) \ge 0}$, and~$\hvarphi$ belongs to~$\cG$ by~\eqref{eq:97}.
  In particular, $\hvarphi$ belongs to~$\Sc$ and thus to~$\rec(\Sp)$ by~\eqref{eq:65} in \cref{p:receding-to-zero-temperature}(1).
  Since~$\varphi_\bullet$ belongs to~$\Sp$ and ${t\varphi = \varphi_\bullet + \hvarphi}$, it follows that~$t \varphi$ belongs to~$\Sp$.
  This proves that~$\varphi$ belongs to~$t^{-1} \Sp$ and completes the proof of~\eqref{eq:92}.

  The first inclusion in~\eqref{eq:93} is a direct consequence of~\eqref{eq:92}, the second of the definition of~$\PTg$, and the last of the first inclusion.
\end{proof}

The proof of \cref{t:persistence-locus} also relies on the following lemma.

\begin{lemm}
  \label{l:persistence-locus}
  Let~$\gamma$ be in~$(0, \min \{\alpha, 1\}]$ and~$\varphi$ in~$\ir(\Sc)$.
  Then, $\varphi$ undergoes a phase transition in temperature and, denoting by~$\beta_*$ the phase transition parameter of~$\varphi$, for every~$\beta$ in~$(\beta_*, +\infty)$, the potential~$\beta \varphi$ belongs to~$\Sp$.
\end{lemm}

\begin{proof}
  By \cref{p:zero-to-positive} there is~$\beta_0$ in~$(0, +\infty)$ such that~$\beta_0 \Sp$ contains~$\varphi$.
  It follows that~$\varphi$ undergoes a phase transition in temperature by \cref{c:pt-in-temperature}.
  Denote by~$\beta_*$ the phase transition parameter of~$\varphi$, so~$\beta_* \varphi$ belongs to~$\partial \Ip$ and ${\beta_* < \beta_0^{-1}}$ by \cref{c:pt-in-temperature}.
  It follows that~$\beta_* \varphi$ belongs to~$\partial \Sp$ by \cref{p:phases}(3).
  Moreover, $\varphi$ belongs to~$\rec(\Sp)$ by \cref{p:receding-to-zero-temperature}(1) and thus, for every~$\beta$ in~$[\beta_0^{-1}, + \infty)$, the potential~$\beta \varphi$ belongs to~$\Sp$.

  To complete the proof of the lemma, let~$\beta$ be in~$(\beta_*, \beta_0^{-1})$ and put
  \begin{equation}
    \label{eq:99}
    \eta_*
    \=
    \frac{\beta_0^{-1} - \beta}{\beta_0^{-1} - \beta_*},
    \eta_0
    \=
    \frac{\beta - \beta_*}{\beta_0^{-1} - \beta_*},
    \text{ and }
    \cU
    \=
    \eta_* (\beta_* \varphi) + \eta_0 \beta_0^{-1} \Sp.
  \end{equation}
  The numbers~$\eta_*$ and~$\eta_0$ are both in~$(0, 1)$, and
  \begin{equation}
    \label{eq:100}
    \eta_* + \eta_0
    =
    1
    \text{ and }
    \eta_* \beta_* + \eta_0 \beta_0^{-1}
    =
    \beta.
  \end{equation}
  In particular, $\cU$ contains~$\beta \varphi$.
  Moreover, since~$\beta_* \varphi$ belongs to~$\partial \Sp$, the potential~$\beta_0^{-1} \varphi$ belongs to~$\Sp$, and~$\cl(\Sp)$ is convex by \cref{p:phases}(2), the inclusion ${\cU \subseteq \cl(\Sp)}$ holds.
  But~$\cU$ is open, so ${\cU \subseteq \Sp}$ and thus~$\beta \varphi$ belongs to~$\Sp$.
  The proof of the lemma is thus complete.
\end{proof}

\begin{proof}[Proof of \cref{t:persistence-locus}]
  The implication ${1 \Rightarrow 2}$ is a direct consequence of either \cite[Corollary~2.11]{0CorInoRiv2504}, or Corollaries~\ref{c:pt-in-temperature} and~\ref{c:states}.
  Since the implication ${2 \Rightarrow 3}$ is trivial, it only remains to prove the implication ${3 \Rightarrow 1}$.

  Suppose that~$\varphi$ belongs to~$\ir(\Sc)$ and let~$\eta$ be in~$(0, 1)$.
  By \cref{l:persistence-locus}, the potential~$\varphi$ undergoes a phase transition in temperature.
  Denote by~$\beta_*$ the phase transition parameter of~$\varphi$ and put ${\varphi_* \= \beta_* \varphi}$.
  Then~$\eta \varphi_*$ belongs to~$\Ip$ by \cref{c:pt-in-temperature} and~$\eta^{-1} \varphi_*$ belongs to~$\Sp$ by \cref{l:persistence-locus}.
  Let~$\cU_*$ be a sufficiently small neighborhood of~$\varphi_*$, such that the open neighborhood~$\beta_*^{-1} \cU_*$ of~$\varphi$ is contained in~$\ir(\Sc)$ and
  \begin{equation}
    \label{eq:101}
    \eta \cU_*
    \subseteq
    \Ip
    \text{ and }
    \eta^{-1} \cU_*
    \subseteq
    \Sp.
  \end{equation}
  Thus, every~$\tvarphi$ in~$\beta_*^{-1} \cU_*$ undergoes a phase transition in temperature by \cref{l:persistence-locus} and satisfies
  \begin{equation}
    \label{eq:102}
    (\eta \beta_*) \tvarphi
    \in
    \Ip
    \text{ and }
    (\eta^{-1} \beta_*) \tvarphi
    \in
    \Sp.
  \end{equation}
  Together with \cref{c:pt-in-temperature}, this implies that the phase transition parameter~$\tbeta$ of~$\tvarphi$ satisfies
  \begin{equation}
    \label{eq:103}
    \eta \beta_*
    <
    \tbeta
    <
    \eta^{-1} \beta_*.
  \end{equation}
  This proves that~$\varphi$ undergoes a persistent phase transition in temperature in~$\Hg$.
\end{proof}

\subsection{Nonpersistent phase transitions in temperature when ${\gamma < \alpha}$}
\label{ss:nonpersistent-pt}

The following proposition gathers results on (non)persistent phase transitions in temperature for~$\gamma$ in~$(0, 1]$ satisfying ${\gamma < \alpha}$.
Recall that the geometric potential~$-\log Df$ undergoes a phase transition in temperature, and ${\ct(-\log Df) = 1}$, see \cref{p:geometric-potential} below.

\begin{prop}
  \label{p:nonpersistent-pt}
  For every~$\gamma$ in~$(0, 1]$ satisfying ${\gamma < \alpha}$, the following properties hold.
  \begin{enumerate}
  \item
    The phase transition in temperature of~$-\log Df$ is nonpersistent in~$\Hg$ and~$\ct$ is discontinuous at~$-\log Df$.
  \item
    The potential~$-\log Df$ belongs to~$\partial \Sc$ and, for each~$\beta$ in~$(0, +\infty)$, the potential~$-\beta \log Df$ belongs to~$\PTg$ if and only if~$\beta$ belongs to~$[1, +\infty)$.
  \item
    The set~$\PTg$ fails to be a co-dimension~$1$ real-analytic subset of~$\Hg$ at~$-\log Df$.
  \end{enumerate}
  In particular, ${\PTg \not\subseteq \ir(\Sc)}$, the function~$\ct$ is discontinuous, and~$\ct^{-1}(1)$ is strictly contained in~$\PTg$.
\end{prop}

This section ends with an example of a real-analytic family of potentials having a phase transition at a potential where no phase transition in temperature takes place (\cref{eq:non-temperature-pt}).
Such an example is only possible if ${\gamma < \alpha}$ by \cref{c:only-temperature-pt}.

The proof of \cref{p:nonpersistent-pt} relies on several preliminary results.
The following lemma is also used in the proof of \cref{t:alpha-rigidity} in~\cref{ss:proof-alpha-rigidity}.

\begin{lemm}
  \label{l:pt-comparison}
  For every~$\gamma$ in~$(0, \min \{ \alpha, 1 \}]$, the function~$\ct$ is lower semicontinuous and satisfies
  \begin{equation}
    \label{eq:104}
    \ct^{-1}(1)
    \subseteq
    \PTg
    \text{ and }
    \ct^{-1}(1) \cap \ir(\Sc)
    =
    \PTg \cap \ir(\Sc).
  \end{equation}
\end{lemm}

\begin{proof}
  To prove that~$\ct$ is lower semicontinuous, let~$\varphi$ be in~$\Hg$ and~$\beta$ in~$(0, \ct(\varphi))$.
  By \cref{c:pt-in-temperature} the potential~$\beta \varphi$ belongs to~$\Ip$.
  Let~$\cV$ be a neighborhood of~$\varphi$ in~$\Hg$ such that ${\beta \cV \subseteq \Ip}$ and let~$\tvarphi$ be in~$\cV$.
  If~$\tvarphi$ undergoes a phase transition in temperature, then ${\ct(\tvarphi) \ge \beta}$ by \cref{c:pt-in-temperature}.
  Otherwise, ${\ct(\varphi) = +\infty}$.
  In all cases, ${\ct(\tvarphi) \ge \beta}$.

  The first inclusion in~\eqref{eq:104} is a direct consequence of \cref{c:pt-in-temperature}, the definition of~$\ct$, and the equality ${\partial \Ip = \PTg}$.
  The second one is given by~\eqref{eq:93} in \cref{p:zero-to-positive}.

  To prove the equality in~\eqref{eq:104}, let~$\varphi$ be in~$\ir(\Sc)$.
  If ${\ct(\varphi) = 1}$, then~$\varphi$ belongs to~$\PTg$ by the first inclusion in~\eqref{eq:104}.
  Moreover, $\varphi$ belongs to~$\Ip$ if ${\ct(\varphi) > 1}$ by \cref{c:pt-in-temperature}, and it belongs to~$\Sp$ if ${\ct(\varphi) < 1}$ by \cref{l:persistence-locus}.
  In both cases~$\varphi$ is outside~$\PTg$.
\end{proof}

Recall that, for~$\gamma$ in~$(0, +\infty)$, the function~$\omega_\gamma$ in~$\Hg$ is given by ${\omega_\gamma(x) = -x^\gamma}$.

\begin{lemm}
  \label{l:nonpersistent-pt}
  For all~$\gamma$ in~$(0, 1]$ satisfying ${\gamma < \alpha}$ and~$\tau$ in~$(-\infty, 0)$,
  \begin{equation}
    \label{eq:105}
    \tau \omega_{\gamma} - \log Df
    \not\in
    \Sc
    \text{ and }
    \ct(\tau \omega_{\gamma} - \log Df)
    =
    +\infty.
  \end{equation}
\end{lemm}

\begin{proof}
  Fix~$\tau$ in~$(-\infty, 0)$, put
  \begin{equation}
    \label{eq:106}
    \varphi_0
    \=
    \tau \omega_{\gamma} - \log Df,
  \end{equation}
  and let~$\varepsilon$ be in~$(0, -\tau)$.
  By the hypothesis ${\gamma < \alpha}$, there is~$N$ in~$\N$ such that, for every~$x$ in~$[0, x_N]$, the inequality ${\varphi_0(x) \ge \varepsilon x^{\gamma}}$ holds.
  Moreover, by~\eqref{eq:18} in \cref{l:neutral-branch} there is~$C_\bullet$ in~$(0, +\infty)$ such that, for every~$i$ in~$\N_0$,
  \begin{equation}
    \label{eq:107}
    x_i^\gamma
    \ge
    \frac{C_\bullet}{i^{\frac{\gamma}{\alpha}}}.
  \end{equation}
  Let~$n$ in~$\N$ be such that ${n \ge N + 1}$ and denote by~$p_n$ the unique periodic point of~$f$ of period~$n$ in~$J_{n - 1}$.
  Then,
  \begin{equation}
    \label{eq:108}
    S_{n - N}(\varphi_0)(p_n)
    \ge
    \varepsilon \sum_{j = 0}^{n - N - 1} f^j(p_n)^{\gamma}
    \ge
    \varepsilon \sum_{i = N + 1}^n x_i^{\gamma}
    \ge
    \varepsilon C_\bullet \sum_{i = N + 1}^n \frac{1}{i^{\frac{\gamma}{\alpha}}},
  \end{equation}
  \begin{equation}
    \label{eq:109}
    S_n(\varphi_0)(p_n)
    \ge
    \varepsilon C_\bullet \sum_{i = N + 1}^n \frac{1}{i^{\frac{\gamma}{\alpha}}} + N \inf_{[0, 1]} \varphi_0,
    \text{ and }
    \lim_{n \to +\infty} S_n(\varphi_0)(p_n)
    =
    +\infty.
  \end{equation}
  This proves that, for every sufficiently large~$n$,
  \begin{equation}
    \label{eq:110}
    \frac{1}{n} S_n(\varphi_0)(p_n)
    >
    0
    =
    \varphi_0(0)
  \end{equation}
  and thus~$\delta_0$ is not a ground state of~$f$ for the potential~$\varphi_0$.
  That is, $\varphi_0$ is outside~$\Sc$.
  Together with \cref{c:persistence-locus}, this yields~\eqref{eq:105}.
\end{proof}

The following proposition gathers well-known properties of the geometric potential.

\begin{prop}
  \label{p:geometric-potential}
  For every~$\beta$ in~$(0, +\infty)$, the inequality ${P(-\beta \log Df) \ge 0}$ holds with equality if and only if~$\beta$ belongs to~$[1, +\infty)$.
  In particular, $-\log Df$ undergoes a phase transition in temperature, and, for every~$\gamma$ in~$(0, \min \{1, \alpha\}]$,
  \begin{equation}
    \label{eq:111}
    -\log Df
    \in
    \PTg
    \text{ and }
    \ct(-\log Df)
    =
    1.
  \end{equation}
\end{prop}

The proof of \cref{p:geometric-potential}, included for completeness, relies on the the following lemma.
For other proofs of \cref{p:geometric-potential}, see \cite{PreSla92} and \cite[Proposition~4.6]{0CorInoRiv2504}, as well as \cite[Theorem~4.3]{CliTho13} when~$\alpha$ belongs to~$(0, 1)$.

\begin{lemm}[\textcolor{black}{\cite[Lemma~2.2]{0CorInoRiv2504}}]
  \label{l:geometric-distortion}
  There are~$\varepsilon_1$ and~$C_1$ in~$(0, +\infty)$ such that, for all~$n$ in~$\N$ and connected component~$J$ of~$f^{-n}(J_0)$,
  \begin{equation}
    \label{eq:112}
    |J|
    \le
    \frac{|J_0|}{(1 + \varepsilon_1 n)^{\frac{1}{\alpha} + 1}}
  \end{equation}
  and, for all~$x$ and~$y$ in~$J$,
  \begin{equation}
    \label{eq:113}
    Df^n(x)
    \ge
    (1 + \varepsilon_1 n)^{\frac{1}{\alpha} + 1}
    \text{ and }
    C_1^{-1}
    \le
    \frac{Df^n(x)}{Df^n(y)}
    \le
    C_1.
  \end{equation}
\end{lemm}

\begin{proof}[Proof of \cref{p:geometric-potential}]
  For every~$\beta$ in~$\R$,
  \begin{equation}
    \label{eq:114}
    P(-\beta \log Df)
    \ge
    h_{\delta_0} - \beta \il \log Df \dd \delta_0
    =
    0.
  \end{equation}
  Moreover, if~$C_1$ is the constant from \cref{l:geometric-distortion}, then, for every~$n$ in~$\N$,
  \begin{equation}
    \label{eq:115}
    1
    \ge
    \sum_{\substack{J \text{ connected} \\ \text{component of } f^{-n}(J_0)}} |J|
    \ge
    C_1^{-1} |J_0| \sum_{x \in f^{-n}(1)} Df^n(x)^{-1}
  \end{equation}
  and thus
  \begin{equation}
    \label{eq:116}
    \limsup_{n \to \infty} \frac{1}{n} \log \sum_{x \in f^{-n}(1)} Df^n(x)^{-1}
    \le
    0.
  \end{equation}
  By~\eqref{eq:20} in \cref{l:pressure} with ${\varphi = -\log Df}$ and ${x = 1}$, the left side of the previous inequality equals~$P(-\log Df)$.
  Thus, for every~$\beta$ in~$[1, +\infty)$,
  \begin{equation}
    \label{eq:117}
    P(-\beta \log Df)
    \le
    P(-\log Df)
    \le
    0
  \end{equation}
  and thus ${P(-\beta \log Df) = 0}$.
  To complete the proof of item~1, it remains to prove that, for every~$\beta$ in~$(0, 1)$, the inequality ${P(-\beta \log Df) > 0}$ holds.
  For each~$n$ in~$\N$, denote by~$K_n$ the maximal invariant set of~$f$ on~$[x_n, 1]$.
  The map~$f$ is uniformly expanding and topologically exact on~$K_n$.
  So, by \textsc{Bowen}'s formula, for every~$\beta$ in~$(0, \HD(K_n))$,
  \begin{equation}
    \label{eq:118}
    P(-\beta \log Df)
    \ge
    P(f|_{K_n}, -\beta \log Df)
    >
    0,
  \end{equation}
  see, for example, \cite[Theorem~9.1.6 and Corollary~9.1.7]{PrzUrb10}.
  Thus, to complete the proof of item~1 it suffices to prove
  \begin{equation}
    \label{eq:119}
    \sup_{n \in \N} \HD(K_n)
    =
    1.
  \end{equation}
  This is given, for example, by \cite[Theorem~19.6.4]{UrbRoyMun22b} applied to the ``CGDMS'' generated by the first return map of~$f$ to~$J_0$.
\end{proof}

\begin{proof}[Proof of \cref{p:nonpersistent-pt}]
  Item~1 is a direct consequence of the equality ${\ct(-\log Df) = 1}$ given by \cref{p:geometric-potential} and of~\eqref{eq:105} in \cref{l:nonpersistent-pt}.

  For item~2, the equality ${\ct(-\log Df) = 1}$ and \cref{c:pt-in-temperature} imply that~$-\log Df$ belongs to~$\Sc$.
  Together with \cref{l:nonpersistent-pt}, this implies that~$-\log Df$ belongs to~$\partial \Sc$.
  For the last assertion, by the equality ${\ct(-\log Df) = 1}$ and \cref{c:pt-in-temperature}, for every~$\beta$ in~$(0, 1)$, the potential~$-\beta \log Df$ belongs to~$\Ip$ and, for every~$\beta$ in~$[1, +\infty)$, the potential~$-\beta \log Df$ is outside~$\Ip$.
  Moreover, \cref{c:pt-in-temperature} and~\eqref{eq:105} in \cref{l:nonpersistent-pt} imply that, for all~$\beta$ in~$(0, +\infty)$ and~$\tau$ in~$(-\infty, 0)$,
  \begin{equation}
    \label{eq:120}
    \ct(\tau \omega_{\gamma} - \beta \log Df)
    =
    \beta^{-1} \ct \left( \frac{\tau}{\beta} \omega_{\gamma} - \log Df \right)
    =
    +\infty
  \end{equation}
  and thus the potential ${\tau \omega_{\gamma} - \beta \log Df}$ belongs to~$\Ip$ by \cref{c:pt-in-temperature}.
  Fixing~$\beta$ in~$[1, +\infty)$ and letting ${\tau \to 0}$, it follows that~$-\beta \log Df$ belongs to~$\partial \Ip$.
  This last set equals~$\PTg$, so~$-\beta \log Df$ belongs to~$\PTg$.

  To prove item~3, recall that~$\PTg$ is closed by definition and note the function ${\R \mapsto \Hg}$ given by ${\beta \mapsto -\beta \log Df}$ is real-analytic.
  If~$\PTg$ were a co-dimension~$1$ real-analytic subset of~$\Hg$, then combining item~2 with \cref{l:identity} would lead to a contradiction.

  For the last assertion of the proposition, for every~$\beta$ in~$(1, +\infty)$, the potential~$-\beta \log Df$ belongs to~$\PTg$ by item~2, and
  \begin{equation}
    \label{eq:121}
    \ct(-\beta \log Df)
    =
    \beta \ct(-\log Df)
    =
    \beta
    >
    1
  \end{equation}
  by \cref{p:geometric-potential}.
  Together with \cref{l:pt-comparison} this proves that~$\ct^{-1}(1)$ is strictly contained in~$\PTg$.
  The remaining assertions are a direct consequence of items~$1$ and~$2$ and of \cref{l:pt-comparison}.
\end{proof}

\begin{exam}
  \label{eq:non-temperature-pt}
  Fix~$\beta_0$ in~$(1, +\infty)$ and let~$(\tvarphi_\tau)_{\tau \in \R}$ be the real-analytic family of potentials in~$\Hg$, defined by
  \begin{equation}
    \label{eq:122}
    \tvarphi_{\tau}
    \=
    \tau \omega_\gamma - \beta_0 \log Df.
  \end{equation}
  The argument below proves that the function ${\tau \mapsto P(\tvarphi_{\tau})}$ is real-analytic on ${\R \setminus \{ 0 \}}$, but fails to be real-analytic at~$0$.
  Since for ${\tau = 0}$,
  \begin{equation}
    \label{eq:123}
    \ct(\tvarphi_0)
    =
    \beta_0^{-1} \ct(-\log Df)
    =
    \beta_0^{-1}
    <
    1
  \end{equation}
  by \cref{p:geometric-potential}, this provides an example of a phase transition at a potential where no phase transition in temperature takes place.
  Note first that, for every~$\tau$ in~$(-\infty, 0)$,
  \begin{equation}
    \label{eq:124}
    \ct(\tvarphi_\tau)
    =
    \beta_0^{-1} \ct \left( \frac{\tau}{\beta_0} \omega_{\gamma} - \log Df \right)
    =
    +\infty
  \end{equation}
  by \cref{l:nonpersistent-pt} and thus the potential~$\tvarphi_\tau$ belongs to~$\Ip$ by \cref{c:pt-in-temperature}.
  Moreover, for ${\tau = 0}$ the potential~$\tvarphi_0$ belongs to~$\PTg$ by \cref{p:nonpersistent-pt}(2).
  It follows that, for every~$\tau$ in~$(0, +\infty)$, the potential~$\varphi_\tau$ belongs to~$\Sp$ by \cref{l:pt-locus}.
  Combined with \cref{t:pt-locus}, this implies that ${\tau \mapsto P(\tvarphi_{\tau})}$ is real-analytic on ${\R \setminus \{ 0 \}}$.
  Furthermore, the function ${\tau \to P(\tvarphi_{\tau})}$ fails to be real-analytic at~$0$ by \cref{l:singular-family}.
\end{exam}

\subsection{Persistence of phase transitions in temperature when ${\gamma = \alpha}$}
\label{ss:persistent-pt}

This section proves the following characterization of potentials undergoing a phase transition in temperature when ${\alpha \le 1}$ and ${\gamma = \alpha}$.
It is the main ingredient in the proof of \cref{t:alpha-rigidity}, which is given in~\cref{ss:proof-alpha-rigidity}.

\begin{prop}
  \label{p:persistent-pt}
  Suppose~$\alpha$ belongs to~$(0, 1]$.
  Then, for every~$\varphi$ in~$\Ha$ the following properties are equivalent.
  \begin{enumerate}
  \item
    The potential~$\varphi$ belongs to~$\ir(\Sca)$.
  \item
    The potential~$\varphi$ undergoes a phase transition in temperature.
  \item
    The measure~$\delta_0$ is the unique ground state of~$f$ for the potential~$\varphi$ and there are~$C$ and~$\delta$ in~$(0, +\infty)$, such that, for every~$n$ in~$\N$,
    \begin{equation}
      \label{eq:125}
      \exp(S_n(\varphi)(x_n) - n \varphi (0))
      \le
      \frac{C}{n^{\delta}}.
    \end{equation}
  \item
    There are constants~$\tC$ and~$\tdelta$ in~$(0, +\infty)$, such that the following property holds.
    For all~$\tvarphi$ in~$\Ha$ close to~$\varphi$, $n$ in~$\N$, and~$x$ in~$(0, 1]$ such that~$f^n(x)$ belongs to~$J_0$,
    \begin{equation}
      \label{eq:126}
      \exp(S_n(\tvarphi)(x) - n \tvarphi(0))
      \le
      \frac{\tC}{n^{\tdelta}}.
    \end{equation}
  \end{enumerate}
\end{prop}

The proof of this proposition relies on a couple of lemmas.

\begin{lemm}
  \label{l:Holder-distortion}
  If~$\alpha$ belongs to~$(0, 1]$, then there is~$D$ in~$(0, +\infty)$ such that, for every~$\varphi$ in~$\Ha$, the following property holds.
  For all~$n$ in~$\N$, connected component~$J$ of~$f^{-n}(J_0)$, and~$x$ and~$x'$ in~$J$,
  \begin{equation}
    \label{eq:127}
    |S_n(\varphi)(x) - S_n(\varphi)(x')|
    \le
    D |\varphi|_{\alpha}.
  \end{equation}
\end{lemm}

\begin{proof}
  Let~$\varepsilon_1$ be the constant from \cref{l:geometric-distortion} and put
  \begin{equation}
    \label{eq:128}
    D
    \=
    |J_0|^\alpha \sum_{k = 1}^{+\infty} \frac{1}{(1 + \varepsilon_1 k)^{1 + \alpha}}.
  \end{equation}

  Let~$n$, $x$, $x'$, and~$J$ be as in the statement of the lemma.
  By \cref{l:geometric-distortion}, for every~$j$ in~$\{0, \ldots, n - 1 \}$,
  \begin{equation}
    \label{eq:129}
    |f^j(J)|
    \le
    \frac{|J_0|}{(1 + \varepsilon_1 (n - j))^{{\frac{1}{\alpha}} + 1}}.
  \end{equation}
  Hence
  \begin{equation}
    \label{eq:130}
    |S_n(\varphi)(x) - S_n(\varphi)(x')|
    \le
    |\varphi|_{\alpha} \sum_{j = 0}^{n - 1} |f^j(x) - f^j(x')|^\alpha
    \le
    |\varphi|_{\alpha} \sum_{j = 0}^{n - 1} |f^j(J)|^\alpha
    \le
    D |\varphi|_{\alpha}.
    \qedhere
  \end{equation}
\end{proof}

\begin{lemm}
  \label{l:Z_1-criterion}
  Suppose~$\alpha$ belongs to~$(0, 1]$ and let~$D$ be the constant from \cref{l:Holder-distortion}.
  If~$\varphi$ in~$\Ha$ satisfies
  \begin{equation}
    \label{eq:131}
    \sum_{n = 1}^{+\infty} \exp(S_n(\varphi)(x_n) - n \varphi(0))
    >
    \exp(D |\varphi|_\alpha),
  \end{equation}
  then ${P(\varphi) > \varphi(0)}$.
\end{lemm}

\begin{proof}
  Consider the formal power series~$\Xi$ and~$\Phi$ in the variable~$s$, defined by
  \begin{align}
    \label{eq:132}
    \Xi(s)
    & \=
      \sum_{n = 1}^{+\infty} \sum_{x \in f^{-n}(1)} \exp(S_n(\varphi)(x) - n \varphi(0)) s^n
      \intertext{and}
      \Phi(s)
    & \=
      \exp(-D |\varphi|_\alpha) \sum_{n = 1}^{+\infty} \exp(S_n(\varphi)(x_n) - n \varphi(0)) s^n.
  \end{align}
  By~\eqref{eq:20} in \cref{l:pressure}, the convergence radius~$R$ of~$\Xi(s)$ satisfies
  \begin{equation}
    \label{eq:133}
    R
    =
    \exp(\varphi(0) - P(\varphi)).
  \end{equation}
  To prove ${R < 1}$, note that~$f|_{J_0}^{-1}$ induces a bijection between the iterated preimages of~$1$ by~$f$ and the iterated preimages of~$1$ by~$F$.
  Furthermore, for all~$n$ in~$\N$ and~$x$ in~$f^{-n}(1)$, the following estimate holds.
  Put ${y \= f|_{J_0}^{-1}(x)}$ and let~$k$ in~$\N$ be such that
  \begin{equation}
    \label{eq:134}
    \sum_{j = 0}^{k - 1} m(F^j(y))
    =
    n + 1.
  \end{equation}
  So, ${F^k(y) = 1}$.
  Applying \cref{l:Holder-distortion} repeatedly yields
  \begin{equation}
    \label{eq:135}
    S_n(\varphi)(x)
    \ge
    S_{n + 1}(\varphi)(y) - \sup_{J_0} \varphi
    \ge
    \sum_{j = 0}^{k - 1} (S_{m(F^j(y))}(\varphi)(x_{m(F^j(y))}) - D |\varphi|_{\alpha}) - \sup_{J_0} \varphi.
  \end{equation}
  It follows that the power series in~$s$,
  \begin{equation}
    \label{eq:136}
    \exp \left( \varphi(0) - \sup_{J_0} \varphi \right) (\Phi(s) + \Phi(s)^2 + \cdots),
  \end{equation}
  has coefficients smaller than or equal to the corresponding coefficients of~$\Xi(s)$.
  In particular, the convergence radius~$\tR$ of~\eqref{eq:136} satisfies ${R \le \tR}$.
  Moreover, the hypothesis~\eqref{eq:131} together with an evaluation~$\Phi(s)$ at ${s = 1}$ yields ${\Phi(1) > 1}$.
  It thus follows that there is~$s_0$ in~$(0, 1)$, such that~$\Phi(s_0)$ is finite and ${\Phi(s_0) \ge 1}$.
  Together with~\eqref{eq:133}, this implies
  \begin{equation}
    \label{eq:137}
    \exp(\varphi(0) - P(\varphi))
    =
    R
    \le
    \tR
    \le
    s_0
    <
    1
    \text{ and }
    P(\varphi)
    >
    \varphi(0).
    \qedhere
  \end{equation}
\end{proof}

\begin{proof}[Proof of \cref{p:persistent-pt}]
  The implication 1~$\Rightarrow$~2 is a direct consequence of \cref{t:persistence-locus}.
  The proofs of the implications 2~$\Rightarrow$~3, 3~$\Rightarrow$~4, and 4~$\Rightarrow$~1 are below.
  Let~$D$ be the constant from \cref{l:Holder-distortion}.
  So, by~\eqref{eq:18} in \cref{l:neutral-branch}, there is~$C_\bullet$ in~$(0, +\infty)$ such that, for every~$i$ in~$\N_0$,
  \begin{equation}
    \label{eq:138}
    x_i^\alpha
    \ge
    \frac{C_\bullet}{i}.
  \end{equation}
  For each continuous potential ${\psi \: [0, 1] \to \R}$ and~$n$ in~$\N$, write
  \begin{equation}
    \label{eq:139}
    \zeta_n(\psi)
    \=
    \exp(S_n(\psi)(x_n) - n \psi(0)).
  \end{equation}

  \partn{2 $\Rightarrow$ 3}
  Suppose~$\varphi$ undergoes a phase transition in temperature.
  Then there is~$\beta$ in~$(0, +\infty)$ such that ${P(\beta \varphi) = \beta \varphi(0)}$ (\cref{c:pt-in-temperature}) and~$\delta_0$ is the unique ground state of~$f$ for the potential~$\varphi$ (\cref{c:states}).
  So, by \cref{l:Z_1-criterion} with~$\varphi$ replaced by~$\beta \varphi$,
  \begin{equation}
    \label{eq:140}
    \sum_{n = 1}^{+\infty} \zeta_n(\varphi)^\beta
    \le
    \exp(D \beta |\varphi|_\alpha).
  \end{equation}
  Moreover, by~\eqref{eq:31} with~$\gamma = \alpha$ and ${\psi = \varphi}$ and by~\eqref{eq:138}, for all~$k$ and~$\ell$ in~$\N$,
  \begin{equation}
    \label{eq:141}
    S_{\ell}(\varphi)(x_{k + \ell}) - \ell \varphi(0)
    \ge
    -|\varphi|_\alpha \sum_{j = 0}^{\ell - 1} x_{k + \ell - j}^{\alpha}
    \ge
    - |\varphi|_\alpha C_\bullet \sum_{i = k + 1}^{k + \ell} \frac{1}{i}
    \ge
    - |\varphi|_\alpha C_\bullet \log \left( \frac{k + \ell}{k} \right).
  \end{equation}
  If ${\ell \le k}$, then
  \begin{equation}
    \label{eq:142}
    \zeta_{k + \ell}(\varphi)
    \ge
    \exp(-(C_\bullet \log 2)|\varphi|_\alpha) \zeta_k(\varphi).
  \end{equation}
  Combined with~\eqref{eq:140}, this yields
  \begin{equation}
    \label{eq:143}
    \exp( \beta D |\varphi|_\alpha )
    \ge
    \sum_{n = 1}^{+\infty} \zeta_n(\varphi)^\beta
    \ge
    \sum_{n = k + 1}^{2k} \zeta_n(\varphi)^\beta
    \ge
    k \exp(-(\beta C_\bullet \log 2)|\varphi|_\alpha) \zeta_k(\varphi)^\beta
  \end{equation}
  and thus
  \begin{equation}
    \label{eq:144}
    \zeta_k(\varphi)
    \le
    \frac{\exp((D + C_\bullet \log 2) |\varphi|_\alpha)}{k^{\frac{1}{\beta}}}.
  \end{equation}

  \partn{3 $\Rightarrow$ 4}
  Suppose item~3 holds, put
  \begin{equation}
    \label{eq:145}
    \tdelta
    \=
    \frac{\delta}{3}
    \text{ and }
    \tD
    \=
    \exp\left( D(|\varphi|_\alpha + C_\bullet^{-1} \tdelta)) \right),
  \end{equation}
  and let~$m_0$ in~$\N$ be sufficiently large so that, for every~$m$ in~$\N$ satisfying ${m \ge m_0}$,
  \begin{equation}
    \label{eq:146}
    \exp \left( |\varphi|_\alpha + C_\bullet^{-1} \tdelta \right) \tD \frac{C \exp(\tdelta)}{(m - 1)^{2\tdelta}}
    \le
    \frac{\tD^{-1}}{(m + 1)^{\tdelta}}.
  \end{equation}
  Furthermore, put
  \begin{equation}
    \label{eq:147}
    \iota
    \=
    \sup \left\{ \il \varphi \dd \nu \: \nu \in \sM, \supp(\nu) \subseteq [x_{m_0}, 1] \right\}
    \text{ and }
    \kappa
    \=
    \exp(\iota - \varphi(0)).
  \end{equation}
  The supremum above is attained, so the hypothesis that~$\delta_0$ is the unique ground state of~$f$ for the potential~$\varphi$ implies
  \begin{equation}
    \label{eq:148}
    \iota
    <
    \varphi(0)
    \text{ and }
    \kappa
    <
    1.
  \end{equation}
  Choose~$\tkappa$ in~$(\kappa, 1)$ and let~$\tK$ in~$(0, +\infty)$ be sufficiently large, so that, for every~$m$ in~$\N$,
  \begin{equation}
    \label{eq:149}
    \tkappa^m
    \le
    \frac{\tK}{(m + 1)^{\tdelta}}.
  \end{equation}

  Let~$\Delta$ in~$\Ha$ be such that
  \begin{equation}
    \label{eq:150}
    |\Delta|_{\alpha}
    <
    \min \left\{ \log \frac{\tkappa}{\kappa}, C_\bullet^{-1} \tdelta \right\}
  \end{equation}
  and put ${\tvarphi \= \varphi + \Delta}$.
  For every measure~$\nu$ in~$\sM$ satisfying ${\supp(\nu) \subseteq [x_{m_0}, 1]}$,
  \begin{equation}
    \label{eq:151}
    \exp \left( \il \tvarphi \dd \nu - \tvarphi(0) \right)
    \le
    \kappa \exp \left( \il \Delta \dd \nu - \Delta(0) \right)
    \le
    \tkappa.
  \end{equation}
  Moreover, by the definition of~$\tD$ in~\eqref{eq:145},
  \begin{equation}
    \label{eq:152}
    \exp(D |\tvarphi|_\alpha)
    \le
    \tD.
  \end{equation}

  The proof of~\eqref{eq:126} relies on a couple of preliminary estimates.
  Let~$x'$ in~$J_0$ and~$k$ in~$\N$ be such that, for every~$\ell$ in~$\{0, \ldots, k - 1\}$, the inequality ${m(F^\ell(x')) \le m_0}$ holds.
  Put ${n' \= \sum_{\ell = 0}^{k - 1} m(F^\ell(x'))}$, so ${f^{n'}(x') = F^k(x')}$.
  Then, for every~$i$ in~$\left\{0, \ldots, n' - 1 \right\}$, the connected component~$J$ of~$f^{-n'}(J_0)$ containing~$x'$ satisfies
  \begin{equation}
    \label{eq:153}
    f^i(J)
    \subseteq
    \bigcup_{m = 0}^{m_0 - 1} J_m
    =
    (x_{m_0}, 1].
  \end{equation}
  In particular, if~$p$ is the unique periodic point of~$f$ of period~$n'$ in~$J$, then the orbit of~$p$ is contained in~$(x_{m_0}, 1]$.
  Thus, by \cref{l:Holder-distortion}, \eqref{eq:151}, and~\eqref{eq:152},
  \begin{equation}
    \label{eq:154}
    \exp(S_{n'}(\tvarphi)(x') - n' \tvarphi(0))
    \le
    \exp(D |\tvarphi|_\alpha) \exp(S_{n'}(\tvarphi) (p) - n' \tvarphi(0))
    \le
    \tD \tkappa^{n'}.
  \end{equation}
  For the second preliminary estimate, \eqref{eq:138} implies, for every~$m$ in~$\N$,
  \begin{equation}
    \label{eq:155}
    S_m(\omega_\alpha)(x_m)
    =
    \sum_{j = 0}^{m - 1} x_{m - j}^{\alpha}
    \ge
    -C_\bullet \sum_{i = 1}^m \frac{1}{i}
    \ge
    -C_\bullet (1 + \log m).
  \end{equation}
  Together with~\eqref{eq:31} with ${\gamma = \alpha}$ and ${\psi = \Delta}$ and with~\eqref{eq:125} and~\eqref{eq:150}, this implies
  \begin{equation}
    \label{eq:156}
    \zeta_m(\tvarphi)
    =
    \zeta_m(\varphi) \zeta_m(\Delta)
    \le
    \frac{C}{m^{\delta}} \exp( - |\Delta|_\alpha S_m(\omega_\alpha)(x_m))
    \le
    \frac{C \exp(C_\bullet |\Delta|_\alpha)}{m^{\delta - C_\bullet |\Delta|_\alpha}}
    <
    \frac{C \exp(\tdelta)}{m^{2 \tdelta}}.
  \end{equation}

  To prove~\eqref{eq:126}, let~$n$ be in~$\N$ and~$x$ in~$(0, 1]$ such that~$f^n(x)$ belongs to~$J_0$.
  Suppose first that~$x$ belongs to~$J_0$.
  Let~$k$ in~$\N$ be such that ${\sum_{\ell = 0}^{k - 1} m(F^\ell(x)) = n}$, so ${F^k(x) = f^n(x)}$.
  For each~$\ell$ in~$\{0, \ldots, k - 1\}$, put
  \begin{equation}
    \label{eq:157}
    z_\ell
    \=
    F^{\ell}(x)
    \text{ and }
    m_\ell
    \=
    m(z_\ell).
  \end{equation}
  If for every~$\ell$ in~$\{0, 1, \ldots k - 1\}$ the inequality ${m_\ell \le m_0}$ holds, then by~\eqref{eq:149} and~\eqref{eq:154},
  \begin{equation}
    \label{eq:158}
    \exp(S_n(\tvarphi)(x) - n \tvarphi(0))
    \le
    \tD \tkappa^n
    \le
    \tD \frac{\tK}{(n + 1)^{\tdelta}}.
  \end{equation}
  Suppose there is~$\ell$ in~$\{0, 1, \ldots k - 1\}$ satisfying ${m(F^\ell(x)) \ge m_0 + 1}$.
  Let~$(\ell_j)_{j = 1}^s$ be the increasing sequence of all such.
  Then, for each~$j$ in~$\{1, \ldots, s\}$, the point~$f(z_{\ell_j})$ belongs to~$J_{m_{\ell_j} - 1}$ and thus by \cref{l:Holder-distortion}, \eqref{eq:146} with ${m = m_{\ell_j}}$, \eqref{eq:150}, \eqref{eq:152}, and~\eqref{eq:156} with ${m = m_{\ell_j} - 1}$,
  \begin{multline}
    \label{eq:159}
    \exp(S_{m_{\ell_j}}(\tvarphi)(z_{\ell_j}) - m_{\ell_j} \tvarphi (0))
    \le
    \exp \left( \tvarphi(z_{\ell_j}) - \tvarphi(0) \right) \tD \zeta_{m_{\ell_j} - 1}(\tvarphi)
    \\ \le
    \exp \left( |\tvarphi|_\alpha \right) \tD \frac{C \exp(\tdelta)}{(m_{\ell_j} - 1)^{2\tdelta}}
    \le
    \frac{\tD^{-1}}{(m_{\ell_j} + 1)^{\tdelta}}.
  \end{multline}
  If ${s = k}$, then, for every~$\ell$ in~$\{0, \ldots, k - 1\}$, the inequality ${m_\ell \ge m_0 + 1}$ holds and
  \begin{multline}
    \label{eq:160}
    \exp(S_n(\tvarphi)(x) - n \tvarphi (0))
    =
    \prod_{\ell = 0}^{k - 1} \exp(S_{m_\ell}(\tvarphi)(z_\ell) - m_\ell \tvarphi (0))
    \le
    \tD^{-k} \prod_{\ell = 0}^{k - 1} \frac{1}{(m_\ell + 1)^{\tdelta}}
    \\ \le
    \tD^{-k} \frac{1}{(n + 1)^{\tdelta}}.
  \end{multline}
  Suppose ${s < k}$, put ${\ell_0 \= -1}$ and ${\ell_{s + 1} \= k}$, and let~$(j(i))_{i = 1}^t$ be the increasing sequence of all~$j$ in~$\{0, \ldots, s\}$ satisfying ${\ell_{j + 1} \ge \ell_j + 2}$.
  For each~$i$ in~$\{1, \ldots, t\}$, put
  \begin{equation}
    \label{eq:161}
    M_i
    \=
    \sum_{\ell = \ell_{j(i)} + 1}^{\ell_{j(i) + 1} - 1} m_\ell.
  \end{equation}
  So, for every~$\ell$ in~${\{\ell_{j(i)} + 1, \ldots, \ell_{j(i) + 1} - 1\}}$, the inequality ${m_\ell \le m_0}$ holds.
  Furthermore,
  \begin{equation}
    \label{eq:162}
    t \le s + 1
    \text{ and }
    \sum_{i = 1}^t M_i
    =
    n - \sum_{j = 1}^s m_{\ell_j}.
  \end{equation}
  Thus, by~\eqref{eq:154}, for every~$i$ in~$\{1, \ldots, t\}$,
  \begin{equation}
    \label{eq:163}
    \exp(S_{M_i}(\tvarphi)(z_{\ell_{j(i) + 1}}) - M_i \tvarphi (0))
    \le
    \tD \tkappa^{M_i}.
  \end{equation}
  Together with~\eqref{eq:149}, \eqref{eq:159}, and~\eqref{eq:162}, this implies
  \begin{multline}
    \label{eq:164}
    \exp(S_n(\tvarphi)(x) - n \tvarphi (0))
    \\
    \begin{aligned}
      & =
        \prod_{j = 1}^s \exp(S_{m_{\ell_i}}(\tvarphi) (z_{\ell_i}) - m_{\ell_i} \tvarphi (0))
        \cdot \prod_{i = 1}^t \exp(S_{M_i}(\tvarphi)(z_{\ell_{j(i) + 1}}) - M_i \tvarphi (0))
      \\ & \le
           \tD^{t - s} \left( \prod_{j = 1}^s \frac{1}{(m_{\ell_j} + 1)^{\tdelta}} \right) \tkappa^{\sum_{i = 1}^t M_i}
      \\ & \le
           \tD \tK \frac{1}{\left( \sum_{j = 1}^s m_{\ell_j} + 1 \right)^{\tdelta}} \cdot \frac{1}{\left( \sum_{i = 1}^s M_i + 1 \right)^{\tdelta}}
      \\ & \le
           \tD \tK \frac{1}{(n + 1)^{\tdelta}}.
    \end{aligned}
  \end{multline}
  Together with~\eqref{eq:158} and~\eqref{eq:160} this completes the proof that, when~$x$ belongs to~$J_0$,
  \begin{equation}
    \label{eq:165}
    \exp(S_n(\tvarphi)(x) - n \tvarphi (0))
    \le
    \frac{\max \{\tD \tK, \tD^{-1} \}}{(n + 1)^{\tdelta}}.
  \end{equation}
  To treat the general case, put
  \begin{equation}
    \label{eq:166}
    \hC
    \=
    \tD C \exp(\tdelta) \max \{\tD \tK, \tD^{-1} \},
  \end{equation}
  suppose~$x$ is not~$J_0$, and let~$m$ in~$\N$ be such that~$J_m$ contains~$x$.
  Then ${m \le n}$, $f^m(x)$ belongs to~$J_0$, and by \cref{l:Holder-distortion} with~$n$ replaced by~$m$ and with ${x' = x_m}$, \eqref{eq:152}, and~\eqref{eq:156},
  \begin{equation}
    \label{eq:167}
    \exp(S_m(\tvarphi)(x) - m \tvarphi (0))
    \le
    \tD \zeta_m(\tvarphi)
    \le
    \frac{\tD C \exp(\tdelta)}{m^{2 \tdelta}}.
  \end{equation}
  When ${n > m}$, by~\eqref{eq:165} with~$n$ replaced by~$n - m$ and~$x$ replaced by~$f^m(x)$ this implies
  \begin{multline}
    \label{eq:168}
    \exp(S_n(\tvarphi)(x) - n \tvarphi (0))
    \le
    \frac{\tD C \exp(\tdelta)}{m^{2 \tdelta}} \exp(S_{n - m}(\tvarphi)(f^m(x)) - (n - m) \tvarphi (0))
    \\ \le
    \frac{\hC}{m^{2 \tdelta} \cdot (n - m + 1)^{\tdelta}}
    \le
    \frac{\hC}{n^{\tdelta}}.
  \end{multline}
  Together with~\eqref{eq:165} and~\eqref{eq:167}, this completes the proof of item~4 in all cases.

  \partn{4 $\Rightarrow$ 1}
  Suppose there are~$\tC$ and~$\tdelta$ in~$(0, +\infty)$ such that~\eqref{eq:126} holds for every~$\tvarphi$ in~$\Ha$ close to~$\varphi$.
  Fix such~$\tvarphi$ and let~$\nu$ be an ergodic measure in~$\sM$ different from~$\delta_0$.
  Then, $\nu$ charges~$J_0$ and by the \textsc{Birkhoff} ergodic theorem there is~$x_\bullet$ in~$J_0$ such that
  \begin{equation}
    \label{eq:169}
    \lim_{n \to +\infty} \frac{1}{n} S_n(\tvarphi)(x_\bullet)
    =
    \il \tvarphi \dd \nu
  \end{equation}
  and such that there are arbitrarily large~$n$ for which~$f^n(x_\bullet)$ belongs to~$J_0$.
  For every such~$n$, \eqref{eq:126} with ${x = x_\bullet}$ yields
  \begin{equation}
    \label{eq:170}
    S_n(\tvarphi)(x_\bullet) - n \tvarphi(0)
    \le
    -\tdelta \log n + \log \tC.
  \end{equation}
  Together with~\eqref{eq:169}, this implies
  \begin{equation}
    \label{eq:171}
    \il \tvarphi \dd \nu
    \le
    \tvarphi(0).
  \end{equation}
  That is, $\delta_0$ is a ground state of~$f$ for the potential~$\tvarphi$ and thus~$\tvarphi$ belongs to~$\Sca$.
  This proves that~$\Sca$ contains a neighborhood of~$\varphi$ in~$\Ha$ and thus~$\varphi$ belongs to~$\ir(\Sca)$.
\end{proof}

\subsection{Proof of \cref{t:alpha-rigidity}}
\label{ss:proof-alpha-rigidity}
The equivalence of properties~$1$, ${2}$, and~$3$ is a direct consequence of \cref{t:persistence-locus} in~\cref{ss:(non)persistent-pt} and \cref{p:persistent-pt}.

To prove~\eqref{eq:14}, let~$\varphi$ be in~$\PTa$.
Then ${P(\varphi) = \varphi(0)}$ by \cref{t:intermittent-potentials} and~$\varphi$ undergoes a phase transition in temperature by \cref{c:pt-in-temperature}.
This phase transition in temperature is persistent in~$\Ha$, so \cref{t:persistence-locus} implies that~$\varphi$ belongs to~$\ir(\Sca)$.
This proves ${\PTa \subseteq \ir(\Sca)}$.
Together with \cref{l:pt-comparison}, this inclusion implies
\begin{equation}
  \label{eq:172}
  \PTa
  =
  \PTa \cap \ir(\Sca)
  =
  \{ \varphi \in \ir(\Sca) \: \cta(\varphi) = 1 \}.
\end{equation}
Thus, to complete the proof of~\eqref{eq:14} it suffices to prove ${\cta = +\infty}$ on ${\Ha \setminus \ir(\Sca)}$.
To do this, let~$\varphi$ be in ${\Ha \setminus \ir(\Sca)}$.
For every~$\beta$ in~$(0, +\infty)$, the potential~$\beta \varphi$ is outside~$\Spa$ by \cref{p:zero-to-positive} and it is outside~$\PTa$ because ${\PTa \subseteq \ir(\Sca)}$.
Hence, $\beta \varphi$ belongs to~$\Ipa$.
Since this holds for every~$\beta$ in~$(0, +\infty)$, by \cref{c:pt-in-temperature} the potential~$\varphi$ does not undergo a phase transition in temperature and ${\cta(\varphi) = +\infty}$ by definition.

It remains to show that~$\cta$ is continuous.
Since ${\cta = +\infty}$ on ${\Ha \setminus \ir(\Sca)}$, this follows from \cref{c:persistence-locus} and the lower semicontinuity of~$\cta$ given by \cref{l:pt-comparison}.

\printbibliography
\end{document}